\documentclass[11pt,a4paper]{amsart}
\usepackage{amsfonts,amsmath,amssymb,amsthm}
\usepackage{times}

\usepackage{graphics}
\usepackage{graphicx}

\usepackage{float}
\usepackage{color}
\usepackage{subfigure}

\numberwithin{equation}{section}

\renewcommand{\epsilon}{\varepsilon}

\renewcommand{\i}{{\ensuremath{\mathrm{i}}}}

\DeclareSymbolFont{SY}{U}{psy}{m}{n}
\DeclareMathSymbol{\emptyset}{\mathord}{SY}{'306}

\DeclareMathOperator{\const}{const}

\newcommand{\R}{\mathbb{R}}

\newcommand{\C}{\mathbb{C}}
\newcommand{\Z}{\mathbb{Z}}
\newcommand{\N}{\mathbb{N}}

\newcommand{\cH}{{\mathcal H}}

\newcommand{\cW}{{\mathcal W}}

\newcommand{\ii}{\mathrm{i}}

{\bf}{\it}

\newtheorem{theorem}{Theorem}[section]{\bf}{\it}
{\bf}{\it}
{\bf}{\it}
{\bf}{\it}
{\it}{\rm}
\newtheorem{lemma}[theorem]{Lemma}{\bf}{\it}
{\it}{\rm}
\newtheorem{definition}[theorem]{Definition}{\bf}{\it}
{\bf}{\it}
{\bf}{\it}

{\bf}{\it}

{\bf}{\it}

\newtheorem{assumption}{Assumption}

\title[Periodic Solutions]{Existence and stability of periodic solutions in a neural field equation}

\date{\today; File: \textbf{\jobname.tex}}

\author[K.~Kolodina]{Karina Kolodina}
\address{K.~Kolodina, Faculty of Science and Technology,
Norwegian University of Life Sciences, P.O. Box 5003, N-1432 {\AA}s,
Norway}
\email{karina.kolodina@nmbu.no}

\author[V.~Kostrykin]{Vadim Kostrykin}
\address{V.~Kostrykin, FB 08 - Institut f\"{u}r Mathematik,
Johannes Gutenberg-Universit\"{a}t Mainz,
Staudinger Weg 9,
55099 Mainz,
Germany}
\email{kostrykin@mathematik.uni-mainz.de}

\author[A.~Oleynik]{Anna Oleynik}
\address{A.~Oleynik, Faculty of Science and Technology,
Norwegian University of Life Sciences, P.O. Box 5003, 1432 {\AA}s,
Norway}
\curraddr {Department of Mathematics, University of Bergen,
Postboks 7803, 5020 Bergen, Norway
}
\email{anna.oleynik@uib.no}

\subjclass[2000]{45L05, 47H30, 47N60, 47G10, 47B48, 47B35}

\keywords{nonlinear integral equations, sigmoid type nonlinearities, neural field model, periodic solutions, block Laurent operators}

\begin{document}

\maketitle
\begin{abstract}
We study the existence and linear stability of stationary periodic solutions to a neural field model, an intergo-differential equation of the Hammerstein type. Under the assumption that the activation function is a discontinuous step function  and the kernel is decaying sufficiently fast, we formulate necessary and sufficient conditions for the existence of a special class of solutions that we call 1-bump periodic solutions. We then analyze the stability of these solutions by studying the spectrum of the Frechet derivative of the corresponding Hammerstein operator. We prove that the spectrum of this operator agrees up to zero with the spectrum of a block Laurent operator. We show that the non-zero spectrum consists of only eigenvalues and obtain an analytical expression for the eigenvalues and the eigenfunctions. The results are illustrated by multiple examples.
\end{abstract}

\section{Introduction}

The behavior of a single layer of neurons can be modeled by a nonlinear integro-differential equation of the Hammerstein type,
\begin{equation}\label{Amari_model}
\frac{\partial}{\partial t}u(x,t) = -u(x,t) + \int_{\R} \omega(x-y) f(u(y, t)-h)dy.
\end{equation}
Here $u(x,t)$ and $f(u(x,t)-h)$ represent the averaged local activity and the firing rate of neurons at the position $x\in \R$ and time $t>0$, respectively. The parameter $h\in \R$ denotes the threshold of firing and  $\omega(x-y)$ describes a coupling between neurons at positions $x$ and $y$.

The model \eqref{Amari_model} belongs to a special class of models, so called neural field models, where the neural tissue is treated as a continuous structure,  and is often referred to as the Amari model.  Since the original paper by Amari \cite{Amari}, this model has been studied in numerous mathematical papers, for a review see, e.g., \cite{coombes2005waves, ermentrout1998neural} and \cite{coombes2014neural}. In particular, the global existence and uniqueness of solutions to the initial value problem for \eqref{Amari_model} under rather mild assumptions on $f$ and $\omega$ has been proven in \cite{potthast2010existence}.

In \cite{Amari} Amari studied pattern formation in \eqref{Amari_model} for a model under the simplifying assumption that $f$ is the unit step function $H$, and $\omega$ is of the "lateral-inhibitory type", i.e., continuous, integrable and even, with $\omega(0)>0$ and having exactly one positive zero.  In particular, he analyzed the existence and stability of stationary localized solutions, or so called 1-bump solutions, of the fixed point problem
\begin{equation}\label{eq:u=Hu}
u(x)=(\cH u)(x), \quad (\cH u)(x)= \int_{-\infty}^{+\infty} \omega(x-y) f(u(y)-h)dy.
\end{equation}

The equations \eqref{Amari_model} and \eqref{eq:u=Hu} have been studied with respect to various combinations of firing rate functions and connectivity functions, see \cite{coombes2005waves, ermentrout1998analysis, coombes2014neural}.
Common examples of $\omega$  are the exponentially decaying function,
\begin{equation}\label{eq:exp}
     \omega(x)= S e^{-s|x|}, \quad S,s>0,
\end{equation}
 the so-called wizard-hat function,
  \begin{equation}\label{eq:exp-exp}
    \omega(x)=S_1 e^{-s_1|x|} - S_2 e^{-s_2|x|}, \quad S_1>S_2>0, \quad s_1>s_2>0,
  \end{equation}
  and  the periodically modulated function
  \begin{equation}\label{eq:exp-osc}
    \omega(x)= e^{-b|x|}(b \sin (|x|))+\cos(x)), \quad b>0,
  \end{equation}
 see Fig.\ref{Fig:omega}.
 In the paper we impose the following assumptions on $\omega.$
\begin{assumption} \label{as:A}
The connectivity function $\omega$ satisfies the following conditions.
\begin{itemize}
\item[(i)] $\omega(x)=\omega(-x)$\\
\item[(ii)] $\omega(x) \to 0$ as $|x| \to \infty$ and $|\omega(x)| \le C(1+|x|)^{-1-\delta},$ $C,\delta=\const>0.$
\item[(iii)] $\omega \in C_b^{0,1}(\R) \cap L_1 (\R).$
\item[(iv)] $\int_{\R}\omega(x) dx=:h_0 > 0$.
\end{itemize}
\end{assumption}

One can easily check that the functions in \eqref{eq:exp} - \eqref{eq:exp-osc} satisfy Assumption \ref{as:A} and decrease exponentially fast as $|x|\to 0.$
\begin{figure}[H]
  \centering
  \includegraphics[width=7cm]{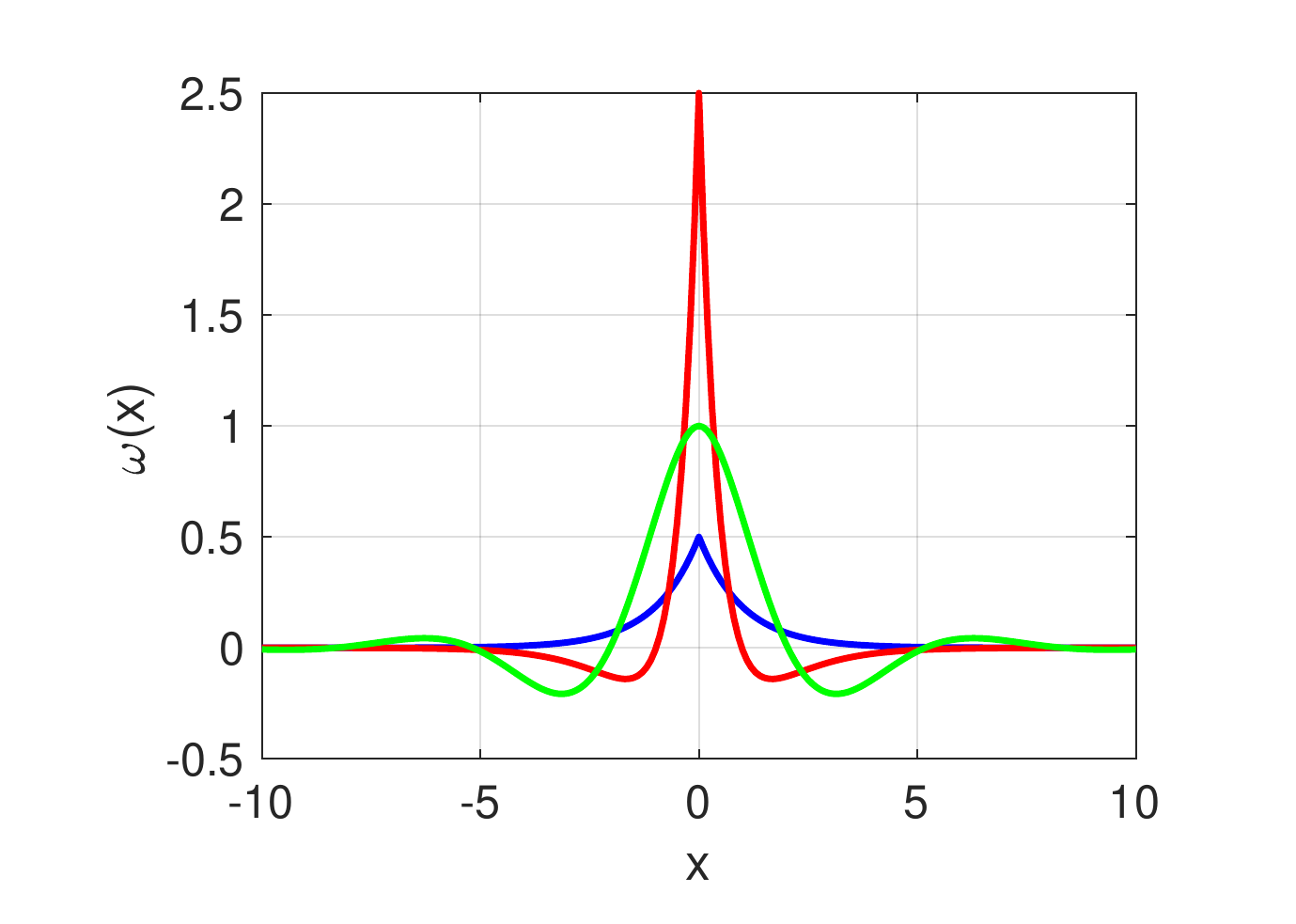}
  \caption{Connectivity functions $\omega(x)$ given by \eqref{eq:exp} with $S=0.5$, $s=1$ (blue curve), \eqref{eq:exp-exp} with $S_1=4$, $s_1=2$, $S_2=1.5$, $s_2=1$ (red curve), and \eqref{eq:exp-osc} with $b=0.5$ (green curve).}\label{Fig:omega}
\end{figure}

The firing rate function $f:\R \to [0,1]$ is usually given as a smooth function of sigmoid shape.  It is often represented by a parameterized function $f(u)=S(\beta u)$, see e.g. \cite{coombes2010neural,Krisner2007PeriodicSolutions,Laing2002MultipleBumps,oleynik2016spatially} where $S(\beta u)$ approaches (in some specific way) the unit step function $H(u)$ as $\beta\to \infty.$ One example of $f(u)$ is
\begin{equation}\label{Eq:FiringRateFun}
f(u)=S(\beta u), \quad
S(u)=\frac{u^p}{u^p+1}H(u), \quad p>1,
\end{equation}
see Fig. \ref{Fig:FiringFun}.

\begin{figure}[H]
  \centering
  \includegraphics[width=7cm]{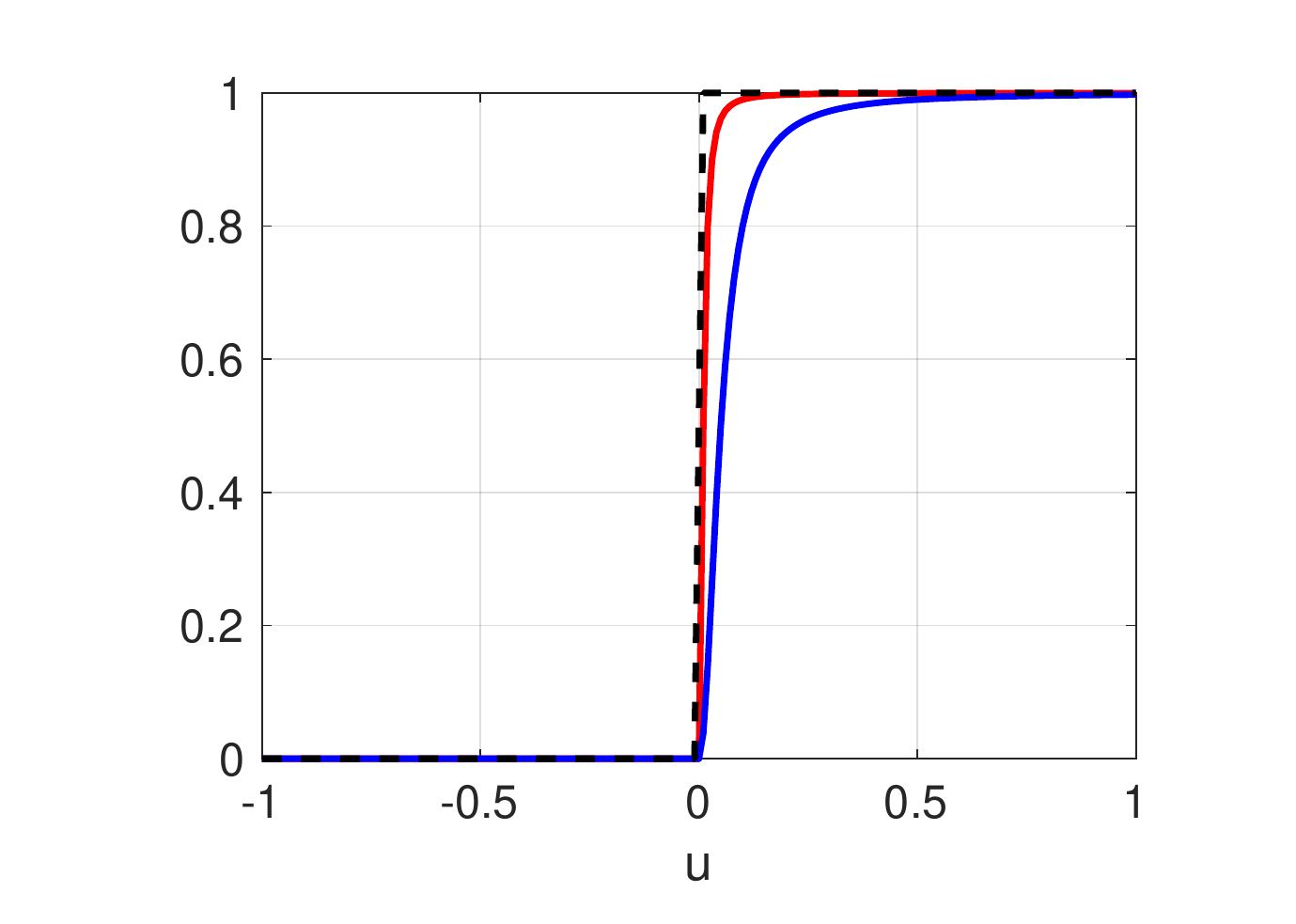}
  \caption{Functions $f(u)=S(\beta u)$, $S$ is as in \eqref{Eq:FiringRateFun}, $p=2$, with $\beta=100$ (red curve) and $\beta=20$ (blue curve) and the unit step function $H(u)$(black dashed line).}\label{Fig:FiringFun}
\end{figure}

Already in his seminal paper Amari conjectured that there must exist periodic stationary solutions in the absence of bump solutions and constant solutions. He however did not pursue a further study of periodic solutions. Of course the absence of other types of stationary solutions is not necessary for periodic solutions to exist. In fact, as in some cases bump solutions can be viewed as a homoclinic orbits of an ordinary differential equation (ODE) with $\omega$ being the Green's function of its linear part, see e.g. \cite{ELVIN2010537}, periodic solutions are very likely to co-exist with the bump solution, see \cite{Shilnikov1965,Shilnikov1970} (in Russian) and \cite{Glendinning1984}, and \cite{DEVANEY1976431}.
In \cite{Bressloff2012, WYLLER200775,ermentrout1998neural}
it has been shown numerically that stable periodic solutions of the two population version of the Amari model exist and emerge from homogeneous solutions via Turing-Hopf bifurcation.
To the best of our knowledge there are no theoretical studies that address the existence of periodic solutions to \eqref{Amari_model} except \cite{Krisner2007PeriodicSolutions}, and no studies on the stability of these solutions.

Krisner in \cite{Krisner2007PeriodicSolutions} studied the existence of periodic solutions to \eqref{Amari_model} with $\omega$ given by \eqref{eq:exp-osc}. In this case, any bounded solution of
\eqref{eq:u=Hu} is a solution of a forth order ODE, see \cite{krisner2004ODE} and can be studied by methods developed for ODEs. Given $f$ as a smooth steep sigmoid function it has been shown that \eqref{Amari_model} has at least two periodic solution under some assumptions on the parameters.  The analysis is however rather cumbersome and is not applicable for general types of $\omega$ as, e.g., \eqref{eq:exp} and \eqref{eq:exp-exp}.
Thus, we would like to proceed in a different way and address the existence of periodic solutions
without reformulating \eqref{eq:u=Hu} as ODEs.

When $f$ is approximated by a step function $H$ it is possible to obtain analytical expressions for some types of stationary solutions and travelling waves, see e.g. chapter 3 in \cite{coombes2014neural} and \cite{CoombesOwen2004}. However, the operator $\cH$ in this case is discontinuous in any classical functional space and thus, classical functional analysis tools such as e.g. generalized Picard-Lindelof theorem or Hartman-Grobman theorem, usually fail. However, many papers still conveniently assume that the model is well-posed on the considered spaces and study the stability of solutions  by first approximation, see \cite{Amari,CoombesOwen2004,BWE2005} and
\cite{OleynikWyllerTetzlaffEinevoll2011}
 just to name a few.


The natural way to overcome this problem is to study the model \eqref{Amari_model} with $f(u)=S(\beta u)$ and only use  the limiting case $f=H$ to gain the knowledge about the existence and stability of solutions for large values of $\beta.$  The approximation of $f=H$ with $f=S(\beta u)$ then must be properly justified.
This has been successfully done for bumps solutions in \cite{oleynik2016spatially,BPW2016a} and \cite{OKS(Unpublished)}.

Our overall aim is to generalize the analysis in the mentioned papers for the periodic 1-bump solutions. In this paper we take the first crucial step towards this direction and study the limiting case $f=H$.

The paper organized as follows:
Section \ref{Sec:Notations} contains the notation we use. In Section \ref{Sec:Existence} we give the definition of $1$-bump periodic solutions and study their existence  by means of the Amari approach.
We formulate necessary and sufficient conditions for the existence of 1-bump periodic solutions and show that for $\omega\geq 0$ there is a unique solution for each period $T>0$.  Section \ref{Sec:Stability} is dedicated to the linear stability of $1$-bump periodic solutions.  We show that the spectrum of the corresponding linearized operator $\cH$ can be obtained as the spectrum of an infinite block Laurent (or bi-infinite block Toeplitz) operator.
We give an analytical expression for the spectrum in terms of the symbol of the Laurent operator and discuss ways how it can be calculated numerically. We prove that the spectrum consists only of eigenvalues and give a formula for calculating eigenfunctions. The results in Section \ref{Sec:Existence} and Section \ref{Sec:Stability} are illustrated for the case of $\omega$ given by \eqref{eq:exp} and \eqref{eq:exp-exp}. Section \ref{Sec:Conclusions} contains conclusions and remarks.


\section{Notations} \label{Sec:Notations}
For the convenience of the readers we give a list of functional spaces and specify other notations we use.

$\bullet$ $S^1$ is the unit circle.

$\bullet$ $\ii$ is the imaginary unit.

$\bullet$ $\overline{z}$ is the complex conjugate of $z\in \C.$

$\bullet$ $cl(\Omega)$ is the closure of a set $\Omega.$

$\bullet$ $\| \cdot\|_{op}$ denotes the operator norm.

$\bullet$ $C^{0,1}_{b}(\R)$ is the space of all Lipschitz continuous bounded functions on $\R$ equipped with the norm
$$\|f\|_{C^{0,1}_b(\R)}=\sup\limits_{x,y\in \R} \frac{|f(x)-f(y)|}{|x-y|}, \quad x\ne y.$$

$\bullet$ $\ell_p^{m}(\Z)$ is the Banach space of sequences with entries from $\R^m$ where $1\leq p \leq \infty$ and $m\in\N$ equipped with the norm
\[
\|x\|_{\ell_{p}^{m}(\Z)}=\left(\sum\limits_{k\in \Z} \|x_k\|^p\right)^{1/p}, \quad 1\leq p<\infty
\]
and
\[
\|x\|_{\ell_{\infty}^{m}(\Z)}=\sup \limits_{k\in\Z}  \|x_k\|, \quad p=\infty,
\]

where $\|\cdot\|$ is any norm in $\R^m.$

$\bullet$ $\ell_{p}^{m\times m}(\Z)$ is the space of sequences where components are matrices $m$ by $m$ on $\R$, equipped with the norm

\[
\|A\|_{\ell_{p}^{m\times m}(\Z)}=\left(\sum\limits_{k\in \Z} \|A_k\|_{op}^p\right)^{1/p}, \quad 1\leq p<\infty
\]
and
\[
\|A\|_{\ell_{\infty}^{m\times m}(\Z)}=\sup \limits_{k\in\Z}  \|A_k\|_{op}, \quad p=\infty.
\]

$\bullet$ $\cW(S^1)$ is the Wiener space of functions defined on $S^1$ (continuous functions whose Fourier coefficients is an $\ell_1(\Z)$ sequence) equipped with the norm
$$\|f\|_{\cW}=\sum\limits_{k\in \Z} |a_k|,$$
where $a_k$ are the Fourier coefficients of $f$.

$\bullet$ $\cW^{m\times m}(S^1)$ is the Wiener space of $m$ by $m$ matrix functions defined on $S^1$ equipped with the norm
$$\|\varphi\|_{op}=\max\limits_{1\leq i\leq m} \sum\limits_{j=1}^{m} \|\varphi_{ij}\|_{\cW}.$$

$\bullet$ $\sigma(L)$ is the spectrum of the linear operator $L$.

$\bullet$ $\rho(L)$ is the resolvent of the linear operator $L$.


\section{Existence of 1-bump periodic solutions}\label{Sec:Existence}
We consider a particular type of periodic solution that we call a 1-bump periodic solution, due to its shape on one considered period, that is, $u(x)\geq h$ on a (connected) interval and $u(x)< h$ otherwise. Krisner in \cite{Krisner2007PeriodicSolutions} proved the existence of the same type of periodic solutions for $\omega$ given in \eqref{eq:exp-osc}. Below we define the class of periodic functions that we intent to consider.

\begin{definition}\label{def:1-bump functions}
Let $h\in \R$, and $u(x)$ be a continuous periodic function defined on $\R$ with a period $T>0.$ We say that $u(x)$ is a \emph{1-bump periodic} function with period $T$, or simply 1-bump periodic, if there is a translation of $u(x),$ say $p(x)=u(x-c),$ with the following properties:
\begin{itemize}
\item[(i)] It has two symmetric intersection, say at $x=\pm a$ with the straight line $y=h,$ i.e.,  $p(\pm a)=h$.\\
\item[(ii)]It lies above $y=h$ for all $x\in (-a,a)$  and below for  $x\in [-T/2, T/2]\setminus[-a,a],$ i.e.,
$p(x)>h$ for  $x\in (-a,a)$ and $p(x)<h$ for $x\in [-T/2, T/2]\setminus[-a,a].$
\end{itemize}

If in addition $u\in C_b^1(\R)$ with $p'(\pm a)\neq 0$ then we say that $u(x)$ is regular.

\end{definition}

We illustrate the definition above in Fig.\ref{Fig:BumpsORnot}

\begin{figure}[H]
  \centering
  \includegraphics[width=7cm]{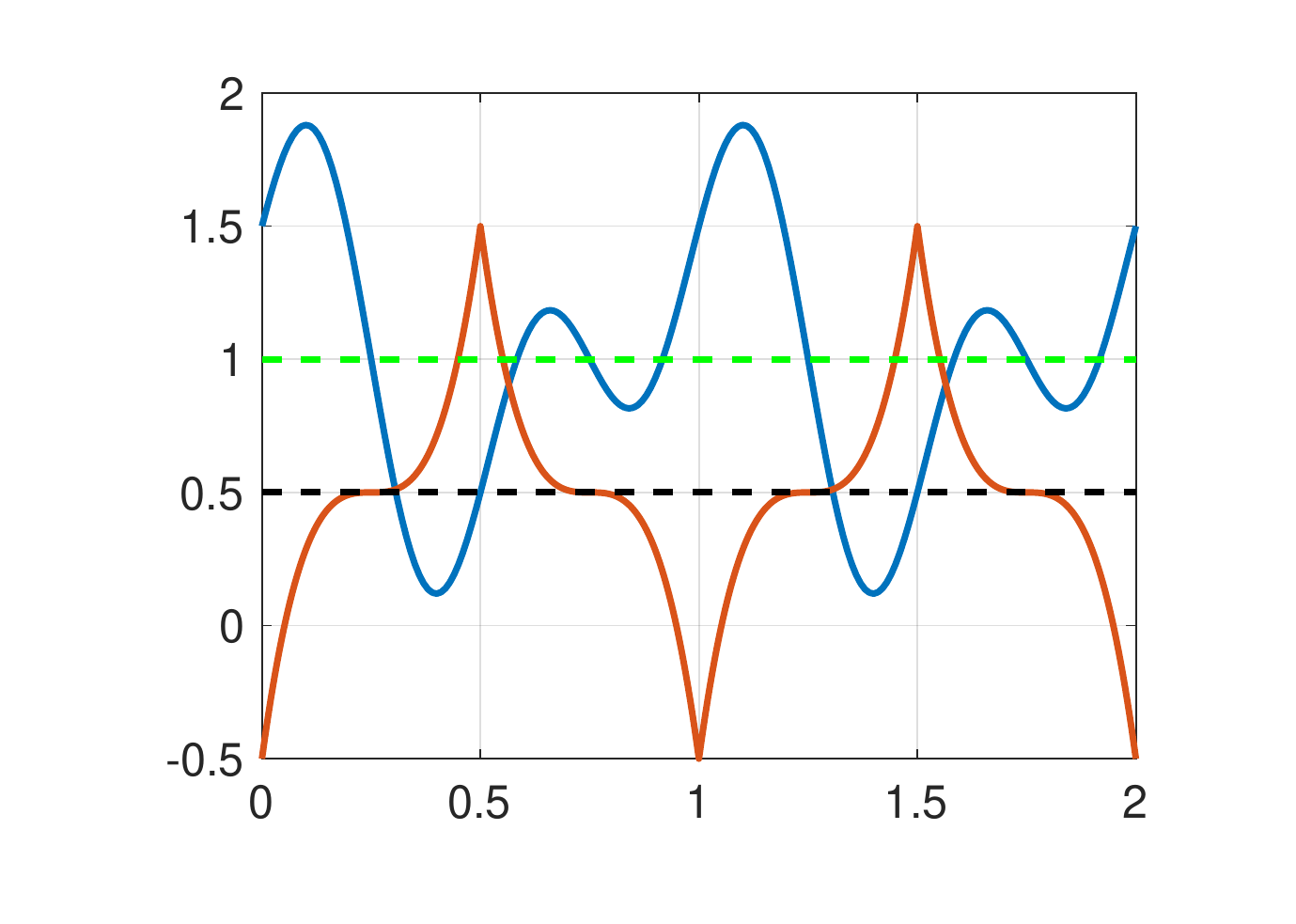}
      \caption{ The function corresponding to the blue curve is the regular 1-bump periodic if $h=0.5$ and is not a 1-bump periodic if $h=1.$ The red curve corresponds to the 1-bump periodic function for both $h=0.5$ and $h=1.$  Here we assume that the functions given by blue and red curves both have period $T=1$.}\label{Fig:BumpsORnot}
\end{figure}
A small perturbation  of a regular 1-bump periodic function in $C_b^{0,1}(\R)$ does not destroy the 1-bump structure of the function. We formulate it as the lemma below.

\begin{lemma}\label{lemma:u_p+v}
 Let $h \in \R$ and $T>0$ be fixed and $p(x)$ be a regular 1-bump periodic function with $p(\pm a)=h,$ $0<a <T/2$. Then there exists $\varepsilon>0$ such that any $v \in B_\varepsilon(u_p):=\{v| : \|v-u_p\|_{C^{0,1}}<\varepsilon \}$ has exactly two intersection with the straight line $y=h$ on each of the intervals $(-T/2+kT, T/2+kT),$ $k \in \Z,$ i.e., there are $a_{\pm}(\varepsilon,k) \in (-T/2+kT, T/2+kT)$   such that $v(a_{\pm}(\varepsilon, k))=h$. Moreover $a_{\pm}(\varepsilon,k) \to \pm a+kT$ as $\varepsilon \to 0$ and
 $v(x)>h$ for $x \in (a_{-}(\varepsilon,k), a_+(\varepsilon,k))$ and $v(x)<h$ for $x \in [-T/2+kT, T/2+kT]\setminus(a_{-}(\varepsilon,k), a_+(\varepsilon,k)).$
 \end{lemma}
\begin{proof}
The proof goes in line with the proof of Lemma 3.6 in \cite{OPW2012}.
\end{proof}

\begin{definition}
A (regular) 1-bump periodic function which is a solution to \eqref{eq:u=Hu} we call a (regular) 1-bump periodic solution to \eqref{Amari_model}.
\end{definition}

We notice that any solution to \eqref{eq:u=Hu} is translation invariant, i.e., if $u(x)$ is a solution to \eqref{eq:u=Hu} then so is $u(x-c)$ for any $c\in \R.$ Thus, without loss of generality we can simply consider $p(x)=u(x)$ in (ii) of Definition \ref{def:1-bump functions}.

Given that $f$ is a unit step function, a 1-bump periodic solution can be expressed as
\begin{equation}
\label{eq:u=int}
u_p(x)=\sum_{k\in \Z}\int_{-a+kT}^{a+kT} \omega(x-y)dy=\sum_{k\in \Z}\int_{-a}^{a} \omega(x-y+Tk)dy
\end{equation}
where $a \in (0, T/2)$ is the root of $u_p(a)=h.$

We notice here that the critical cases $a=0$ and $a=T/2$ correspond to the constant solutions
$u_p(x)=0$ and $u_p(x)=h_0$ where $h_0= \int_\R \omega(y)dy >0.$ This serves as a motivation to consider $h\in (0, h_0).$ Further we will show that for some connectivity functions the condition $h\in (0,h_0)$ is sufficient for the existence of a 1-bump periodic solution.

It is easy to see that the function  in \eqref{eq:u=int} is periodic. Indeed,
\begin{equation*}
u_p(x+T)=\sum_{k\in \Z}\int_{-a}^{a} \omega(x-y+T(k+1))dy=u_p(x).
\end{equation*}
Moreover, due to Assumption \ref{as:A} (i), it is even
\begin{equation*}
\begin{split}
u_p(-x) &= \sum_{k\in\Z}\int_{-a}^a\omega(-x-y+Tk) dy = \sum_{k\in\Z}\int_{-a}^a\omega(-x+y+Tk) dy\\
&= \sum_{k\in\Z}\int_{-a}^a\omega(x-y-Tk) dy = u_p(x).
\end{split}
\end{equation*}

From Assumption \ref{as:A}(ii) we obtain the following
estimate
\begin{equation*}
\max \limits_{\begin{array}{c} x\in [-T/2,T/2]\\ y\in [-a,a]\end{array}} |\omega(x-y+Tk)| \leq C \alpha_k,\quad k\in\Z,
\end{equation*}
where
\begin{equation*}
\alpha_k=\begin{cases}
1, & k=0,\\
(1+T|k|-a-T)^{-1-\delta}, & |k|\geq 1.
\end{cases}
\end{equation*}

Since $\sum\limits_{k\in\Z} \alpha_k$ converges, the series $\sum\limits_{k\in \N}\omega(x+Tk)$ converges absolutely and uniformly on $[-a-T/2,a+T/2].$ Due to periodicity of this series, it converges absolutely and uniformly on any bounded interval to an even periodic function
\begin{equation*}
\omega_p(x;T):=\sum_{k\in\Z} \omega(x-Tk)
\end{equation*}
that has the antiderivative
\begin{equation}
W_p(x;T):=\int_0^x \omega_p(y;T)dy= \sum_{k\in \Z}\int_0^x \omega(y-kT).
 \end{equation}

Using the notations above we obtain
\begin{equation}
\label{eq:u=int2}
u_p(x)=\int\limits_{-a}^{a}\omega_p(x-y;T)dy
\end{equation}
or, equivalently,
\begin{equation}
\label{eq:u=int3}
u_p(x)=W_p(x+a;T)-W_p(x-a;T)
\end{equation}
where $a$ is then given as
\begin{equation}
\label{eq:a}
W_p(2a;T)=h.
\end{equation}

Thus, the procedure of finding 1-bump periodic solutions becomes analogous to the one of finding 1-bump solutions proposed by Amari in \cite{Amari} where instead of  $\omega$ and $W$ we use $\omega_p$ and $W_p,$ respectively. Namely, first we find $a$ from \eqref{eq:a}. Then we verify that the function in \eqref{eq:u=int3} is indeed a 1-bump periodic function. As the function $u_p$ is even and periodic, it is enough to consider the interval $[0,T/2]$.
We summarize this in a theorem.

\begin{theorem} \label{th:NesSufCond}
The function $u_p(x)$ given by \eqref{eq:u=int3} is a periodic solution to \eqref{Amari_model} if and only if the following three conditions hold
\begin{itemize}
\item[(1)] $u_p(a)=h$, or equivalently, $W_p(2a;T)=h,$ for some $0<a<T/2,$
\item[(2)] $u_p(x)>h$ for all $x\in(0,a)$,
\item[(3)] $u_p(x)< h$ for all $x\in (a,T/2]$.
\end{itemize}
\end{theorem}

Similarly as for the bump solutions, it is not generally possible to verify the conditions of the theorem above without additional information about $\omega.$ However, for a particular choice of $\omega$ the verification procedure is rather simple.

Observe that from \eqref{eq:u=int3} $u_p \in C_b^1(\R).$ Then we calculate
\begin{equation} \label{eq:u'}
u'_p(x)= \omega_p(x+a;T)-\omega_p(x-a;T)
\end{equation}
and
\begin{equation} \label{eq:u'(a)}
|u'_p(a)|= \omega_p(0;T)-\omega_p(2a;T).
\end{equation}
Hence, if $u_p$ is a 1-bump periodic solution, $\omega_p(0;T)\geq\omega_p(2a;T)$ must be satisfied.
Then for $\omega\geq 0$ we can simplify conditions of Theorem \eqref{th:NesSufCond}.

\begin{lemma}\label{lem:AnnaVadim}
Let $T>0$ be arbitrary and $\omega$ satisfies Assumption \ref{as:A}. Then for any $0<h<h_0$ the equation $u_p(a)=h$ possesses at least one solution $a \in (0, T/2)$.
If $\omega \ge 0$ and can have only isolated zeros then such $a=a(T)$ is unique and the corresponding $u_p$ is a 1-bump regular periodic solution provided that $\omega_p(2a(T);T) < \omega_p(0;T)$.
\end{lemma}
\begin{proof}
Since the function $W_p(x;T)$ is continuous and $W_p(0;T)=0$ and $W_p(T;T)=h_0>0,$ there is at least one solution to the equation $W_p(2a;T)=h$ with $0<a<T/2$.

Assume now that $\omega\geq 0$ and does not have non isolated zeros.
Then $W_p(x;T)$ is strictly monotone increasing on $[0, T/2]$. Indeed,
\begin{equation*}
\frac{d}{dx} W_p(x;T)= \omega_p(x;T)=\sum_{k\in\Z} \omega(x+Tk)\geq 0
\end{equation*}
and may have only isolated zeros. This implies the uniqueness of $a$ as a function of $T.$
The final statement follows from \eqref{eq:u'(a)} and uniqueness of $a.$
\end{proof}

For more general function $\omega$ number of 1-bump periodic solution may vary with the period. In the next section we give several examples of $\omega,$ $T$ and $h$ for which the solutions do not exists, exists and are unique or non-unique.

\subsection{Examples} \label{Sec:Existence:Ex}

We consider two examples of the connectivity functions given in \eqref{eq:exp} and \eqref{eq:exp-exp} where most of the calculations can be done analytically.

Indeed, for $\omega$ given by  \eqref{eq:exp} we get $h_0= 2S/s$,
\[
\omega_p(x;T)=S\psi(x \bmod T; s) , \,  \mbox{ and } \,
W_p(x)=\frac{2S}{s}\left\lfloor \frac{x}{T} \right\rfloor + S\Psi(x \bmod T; s)
\]
where
\begin{equation} \label{eq:phiPhi:1}
\psi(x;s)=\frac{\exp(-sx)+\exp(-s(T-x))}{1-\exp(-sT)}
\end{equation}
\begin{equation}\label{eq:phiPhi:2}
\Psi(x;s)= \frac{\exp(s(x-T))-\exp(-sx)-\exp(-sT)+1}{s (1-\exp(-sT))},
\end{equation}
 see Fig.\ref{Fig:wpWp1andBump}(a).

From Lemma \ref{lem:AnnaVadim} the equation $W_p(2a;T)=h,$ $h\in (0, h_0)$ possesses a unique solution $0<a(T)<T/2.$  Moreover, $\omega_p(0;T)=2S/(1-\exp(-sT))$ and
$\omega_p(2a;T)=(\exp(-2sa)+\exp(-s(T-2a)))/(1-\exp(-sT))$ and thus, $\omega_p(2a;T)<\omega_p(0;T)$  for any $T>0$ which implies that $u_p(x)$ is a 1-bump regular periodic solution, see Fig.3.

\begin{figure}[H]
  \centering
 
\subfigure[]{\includegraphics[width=7cm]{{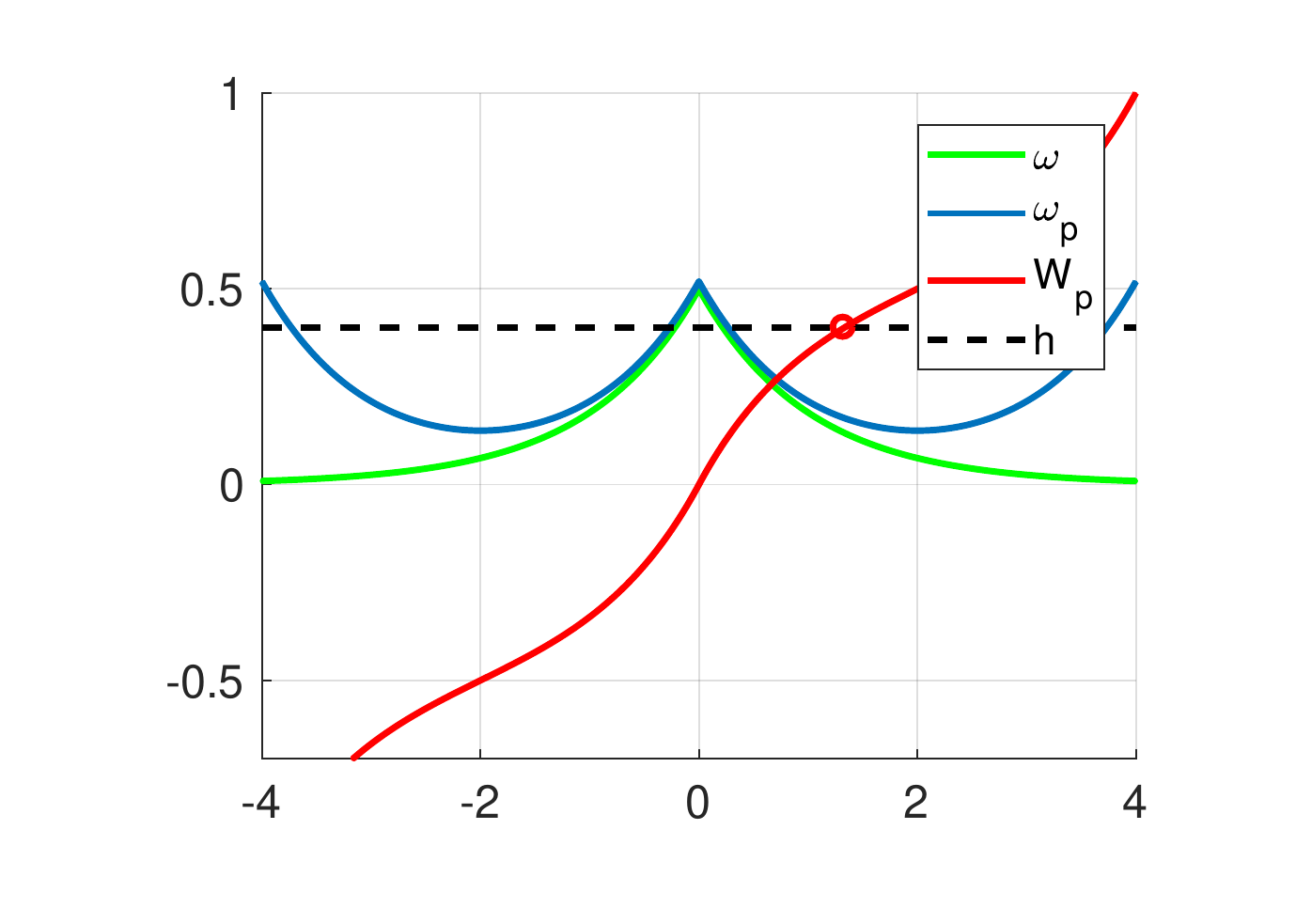}}}
\subfigure[]{\includegraphics[width=7cm]{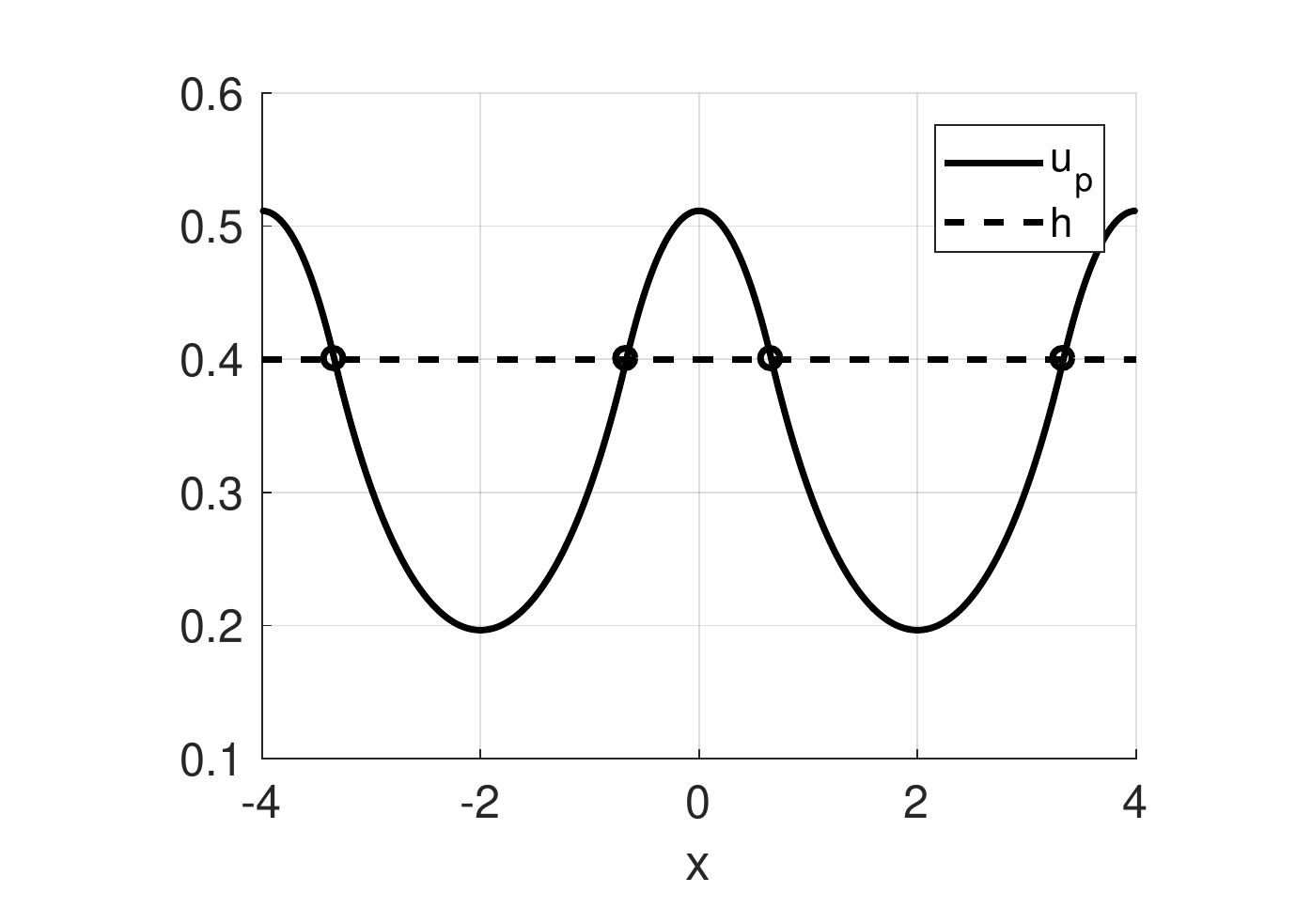}}
  \caption{(a)The function $\omega$ given in \eqref{eq:exp} with $S=0.5$, $s=1$ and the corresponding $\omega_p$ and $W_p$ with $T=4$. The intersection point corresponds to $a=0.6633$ (rounded up to 4 decimals) and $h=0.4.$ (b) 1-periodic bump solution \eqref{eq:u=int3} with $W_p$ as in (a).}\label{Fig:wpWp1andBump}
\end{figure}

For $\omega$ given by \eqref{eq:exp-exp} we find $h_0=2(S_1/s_1-S_2/s_2),$
\begin{equation} \label{eq:wp:2}
\omega_p(x;T)=S_1\psi(x \bmod T; s_1)-S_2 \psi( x\bmod T; s_2)
\end{equation}
 and
\begin{equation} \label{eq:Wp:2}
W_p(x;T)=\left(\frac{2S_1}{s_1}- \frac{2S_2}{s_2}\right) \left\lfloor \frac{x}{T} \right\rfloor + S_1\Psi(x \bmod T; s_1) -S_2\Psi(x \bmod T; s_2)
\end{equation}
 with $\psi$ and $\Psi$ given as in \eqref{eq:phiPhi:1}-\eqref{eq:phiPhi:2}, see Fig. \ref{Fig:wpWp2}.
\begin{figure}[H]
  \centering
  \includegraphics[width=7cm]{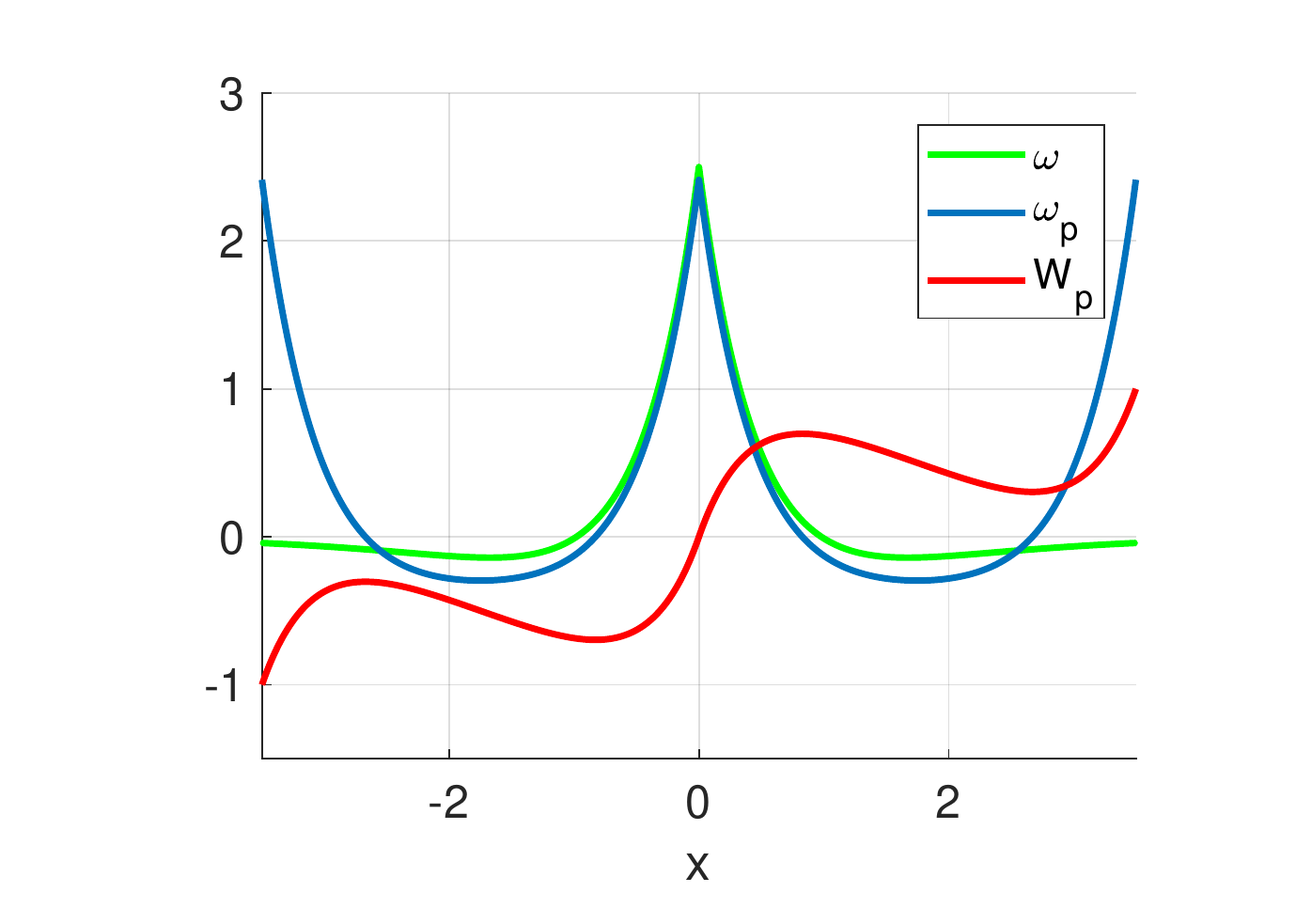}
  \caption{The function $\omega$ given by \eqref{eq:exp-exp} with parameters $S_1=4$, $s_1=2$, $S_2=1.5$, $s_2=1$ and the corresponding $\omega_p$ and $W_p$ , see \eqref{eq:wp:2}-\eqref{eq:Wp:2} with $T=3.5$.}\label{Fig:wpWp2}
\end{figure}

The equation $W_p(2a;T)=h,$ $h\in (0,h_0)$ has one, two, or three solutions depending on $T$. That is for the parameter values $S_1=4,$ $s_1=2,$ $S_2=1.5,$ $s_2=1$ and $h=0.4$, it has one solution for $T<T_1:=2.4997$, two solutions for $T=T_1$ and three solutions for $T>T_1$,  see Fig.\ref{Fig:WpWpWp}.  The value $T_1=2.4997$ is obtained numerically and is rounded up to four decimals. It turns out that all of $u_p$ correspond to 1-periodic bump solutions, see Fig.\ref{Fig:expexpBumps12} -Fig.\ref{Fig:expexpBumps33} .

\begin{figure}[H]
  \centering
  \includegraphics[width=7cm]{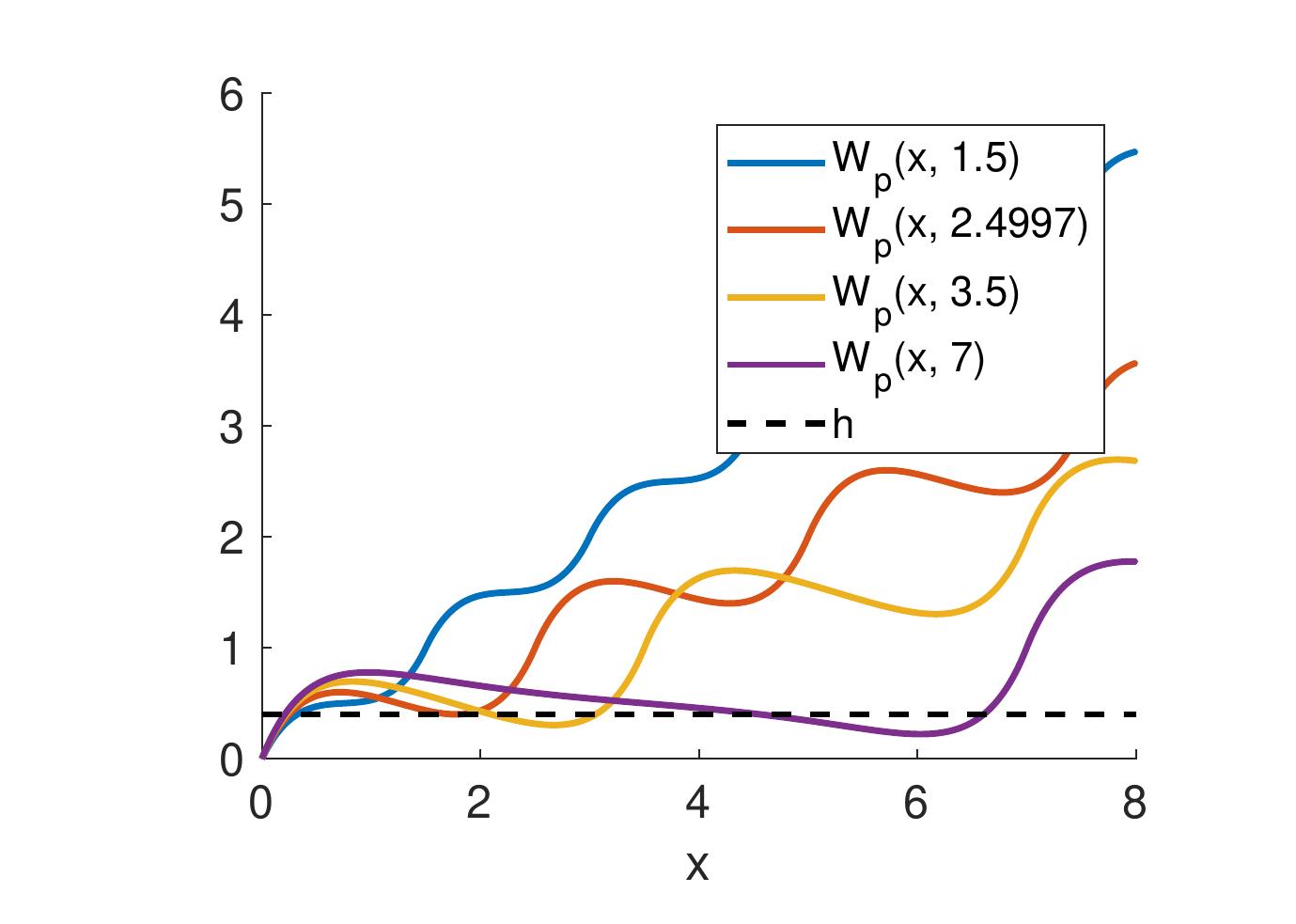}
  \caption{The function $W_p$ in \eqref{eq:Wp:2} with parameters $S_1=4,$ $s_1=2,$ $S_2=1.5$, $s_2=1$ for different periods $T$ and the fixed threshold value $h=0.4$.}\label{Fig:WpWpWp}
\end{figure}

\begin{figure}[H]
  \centering
  \subfigure[]{  \includegraphics[width=7cm]{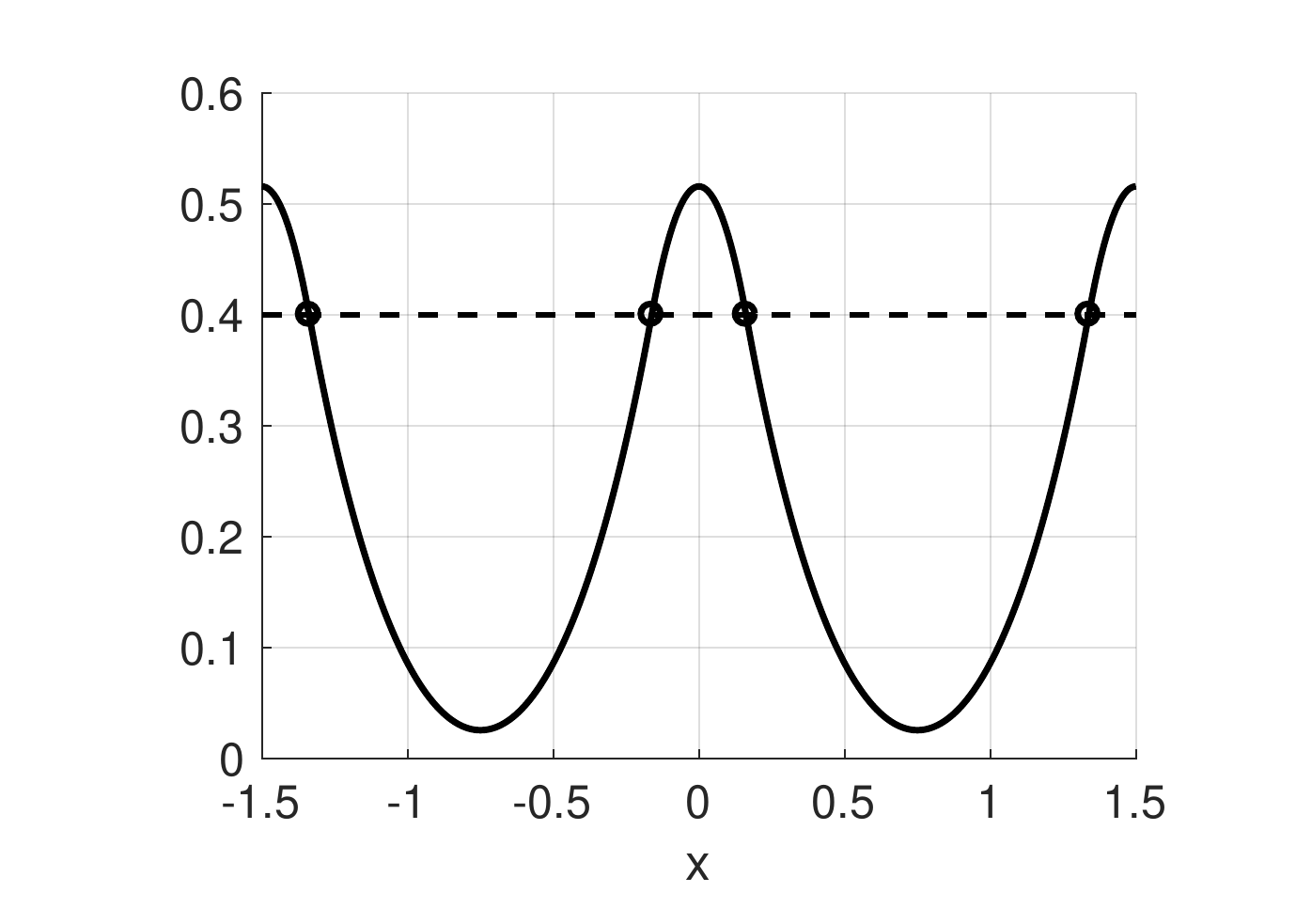}}
  \subfigure[]{  \includegraphics[width=7cm]{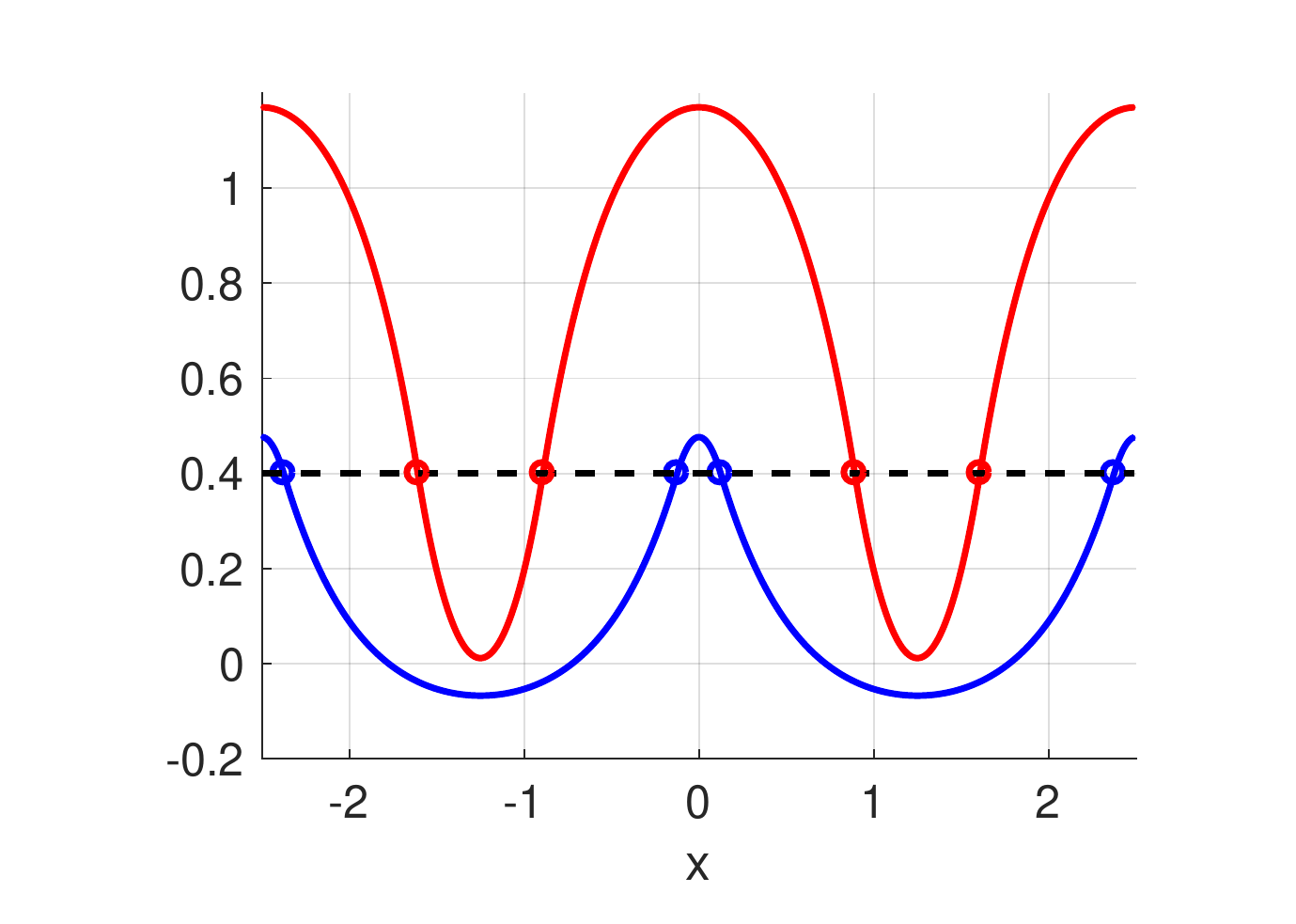}}
    \caption{(a) $1$-bump periodic solutions \eqref{eq:u=int3} with $T=1.5$. The intersection point corresponds to $a=0.1619$ and $h=0.4$. (b) $1$-bump periodic solutions \eqref{eq:u=int3} with $T=2.4997$. The intersection points correspond to $a_1=0.1243$, $a_2=0.8919$ and $h=0.4$. (All the approximated values are rounded up to 4 decimals.)}\label{Fig:expexpBumps12}
\end{figure}

\begin{figure}[H]
  \centering
  \subfigure[]{\includegraphics[width=7cm]{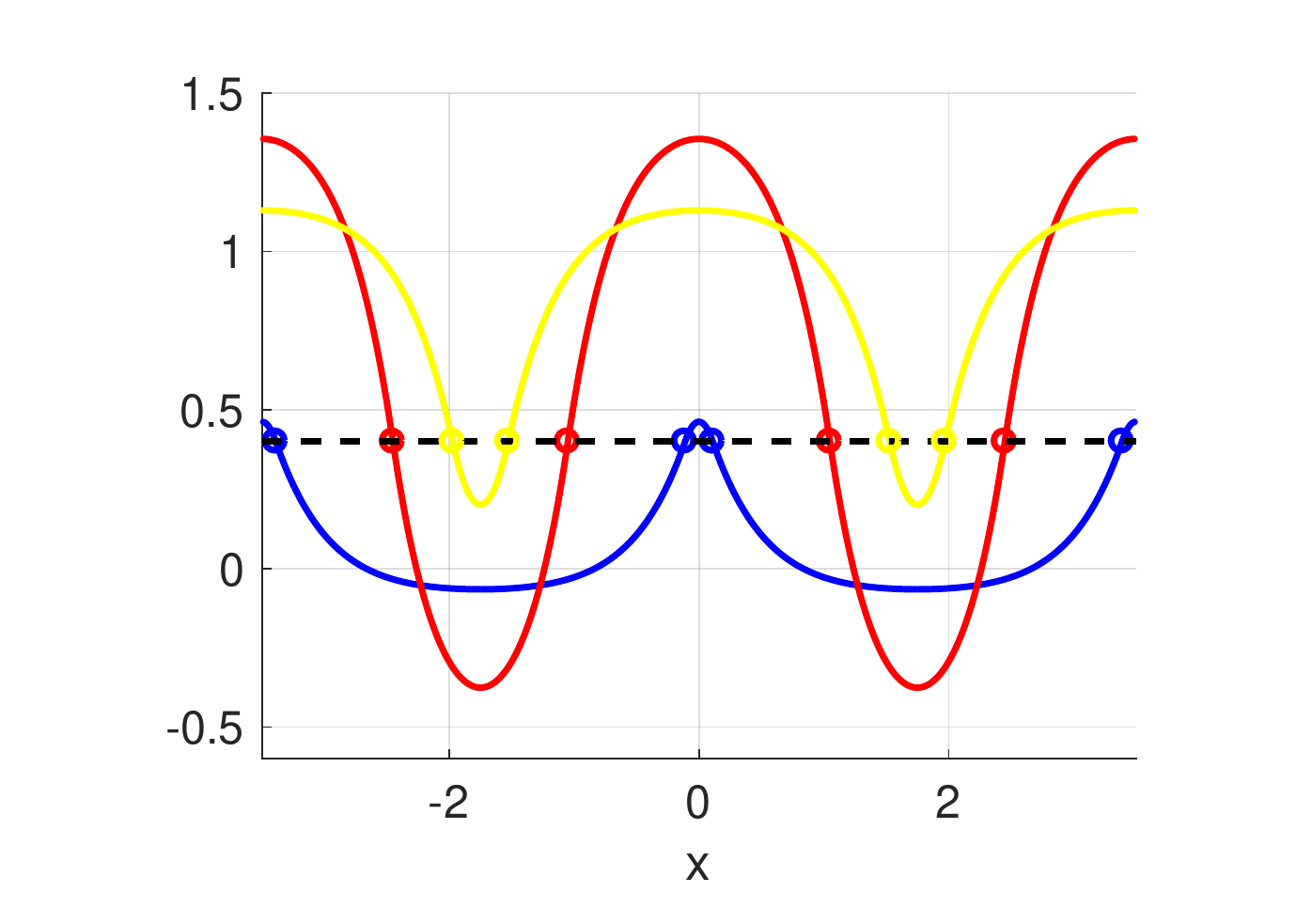}}
  \subfigure[]{\includegraphics[width=7cm]{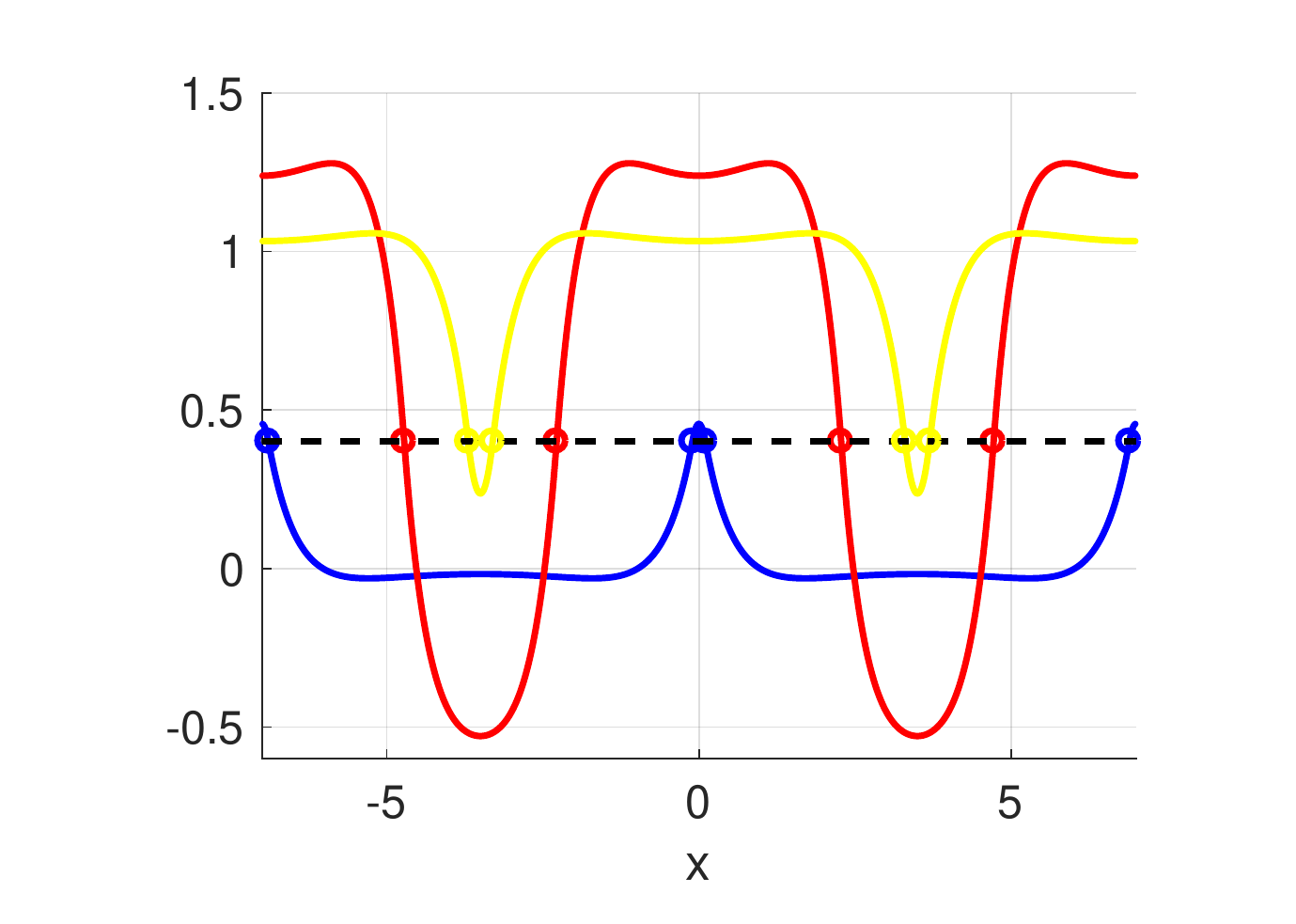}}
    \caption{(a) $1$-bump periodic solutions \eqref{eq:u=int3} with $T=3.5$. The intersection points correspond to $a_1=0.1113$, $a_2=1.0494$ and $a_3=1.5281,$ and $h=0.4$. (b) $1$-bump periodic solutions \eqref{eq:u=int3} with $T=7$. The intersection points correspond to $a_1=0.1046$, $a_2=2.2792$ and $a_3=3.3036$ and $h=0.4$. (All the approximated values are rounded up to 4 decimals.)}\label{Fig:expexpBumps33}
\end{figure}

There are parameters $S_1, S_2,$ and  $s_1, s_2$ that $W_p(2a;T)=h$  have two solutions for $h>h_0$ and some $T>0.$ For example, for $S_1=3,$ $s_1=2,$ $S_2=1.4,$ $s_2=1,$ and $h=0.25$ this situation occurs when $T>2.116$, see Fig. \ref{Fig:X(a)}.  These solutions correspond to the 1-bump periodic solutions, see Fig. \ref{Fig:X(b)}. We however do not aim to study this particular case of the connectivity function in detail. Thus, we will further restrict our attention to the case $h<h_0$, see  Fig. \ref{Fig:WpWpWp}.

\begin{figure}[H]
  \centering
  \includegraphics[width=7cm]{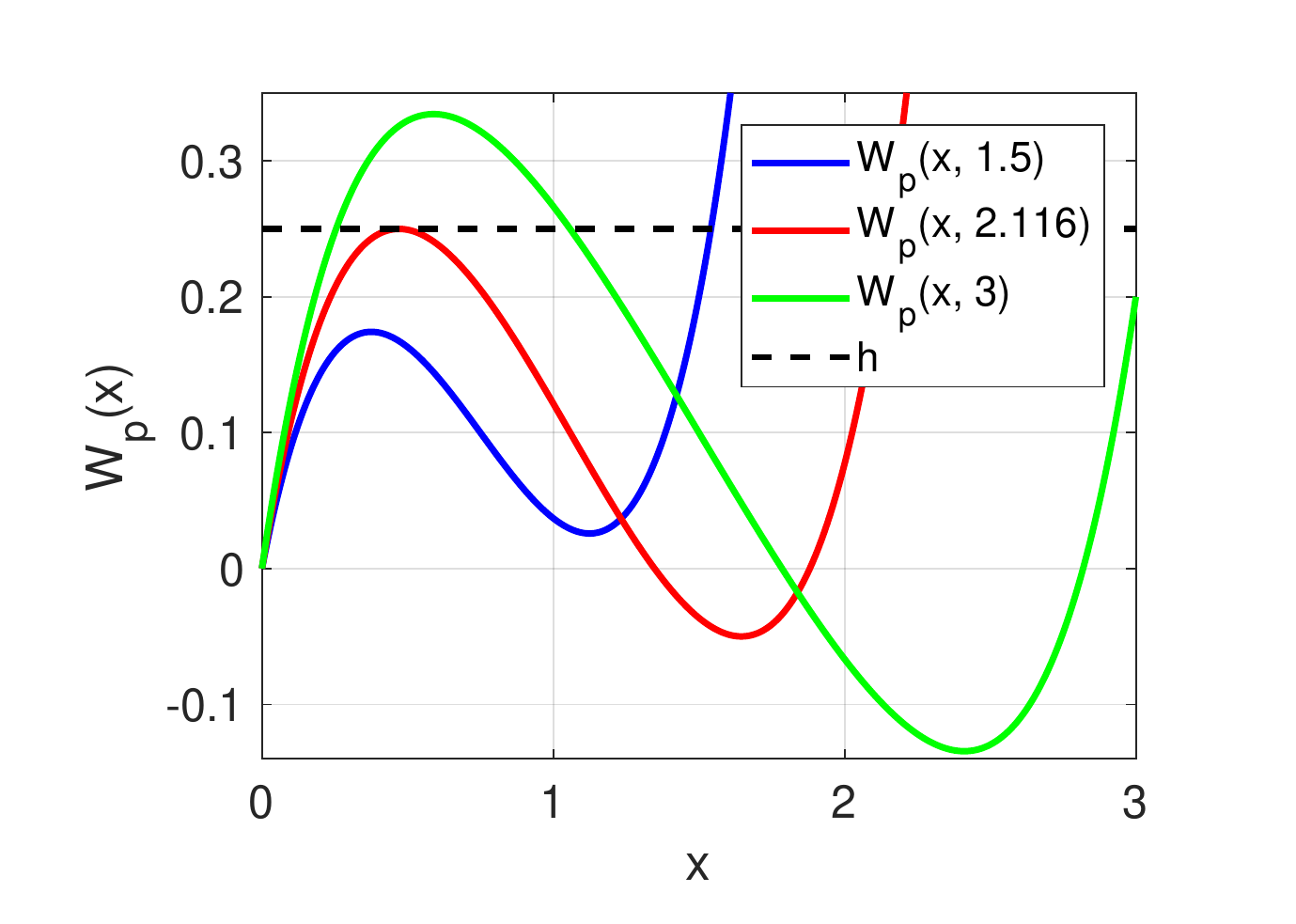}
  \caption{The function $W_p(x;T)$ in \eqref{eq:Wp:2} with parameters $S_1=3,$ $s_1=2,$ $S_2=1.4,$ $s_2=1,$ for different periods $T$ and the fixed threshold value $h=0.25$.}\label{Fig:X(a)}
\end{figure}

\begin{figure}[H]
  \centering
  \subfigure[]{  \includegraphics[width=7cm]{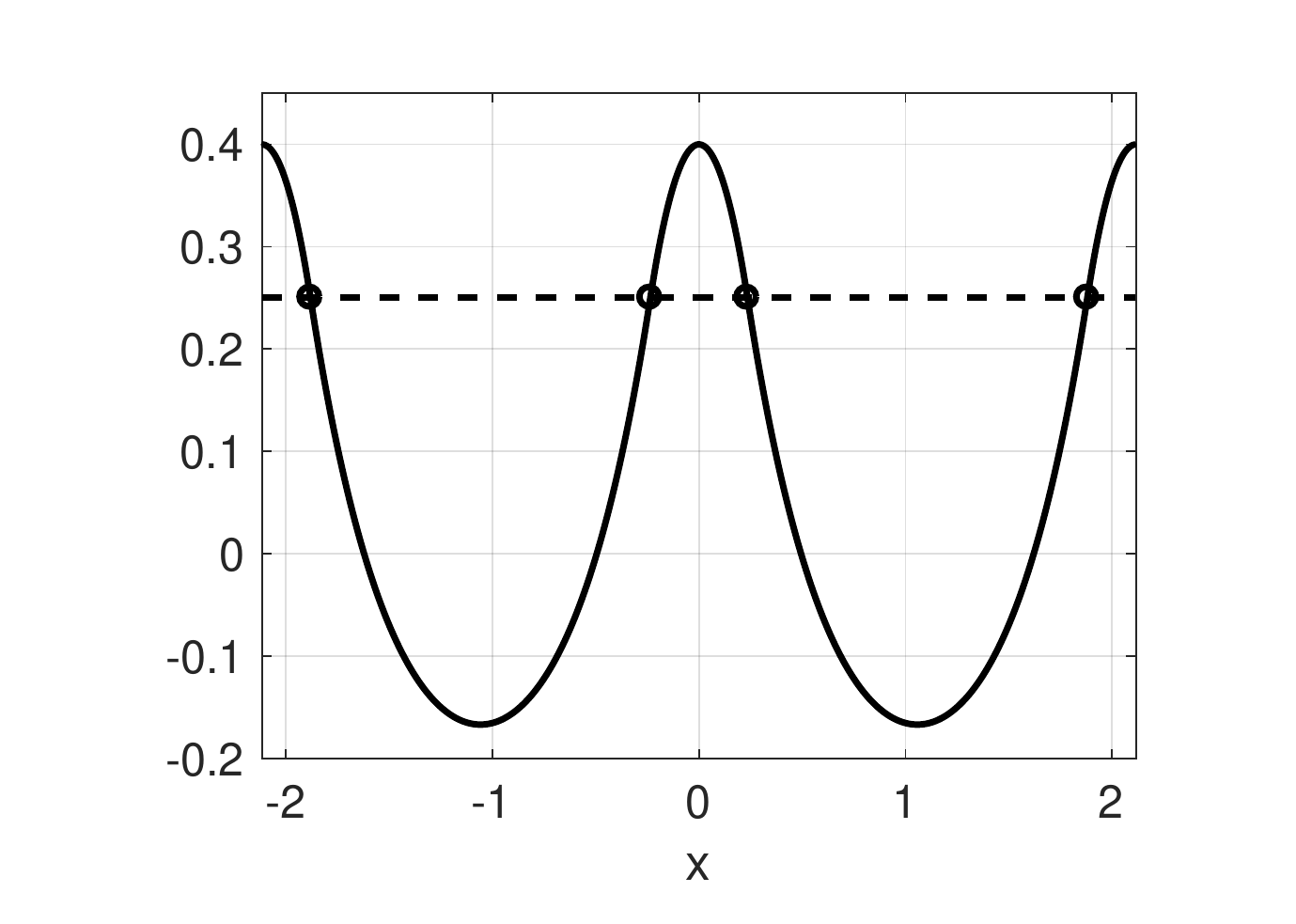}}
  \subfigure[]{\includegraphics[width=7cm]{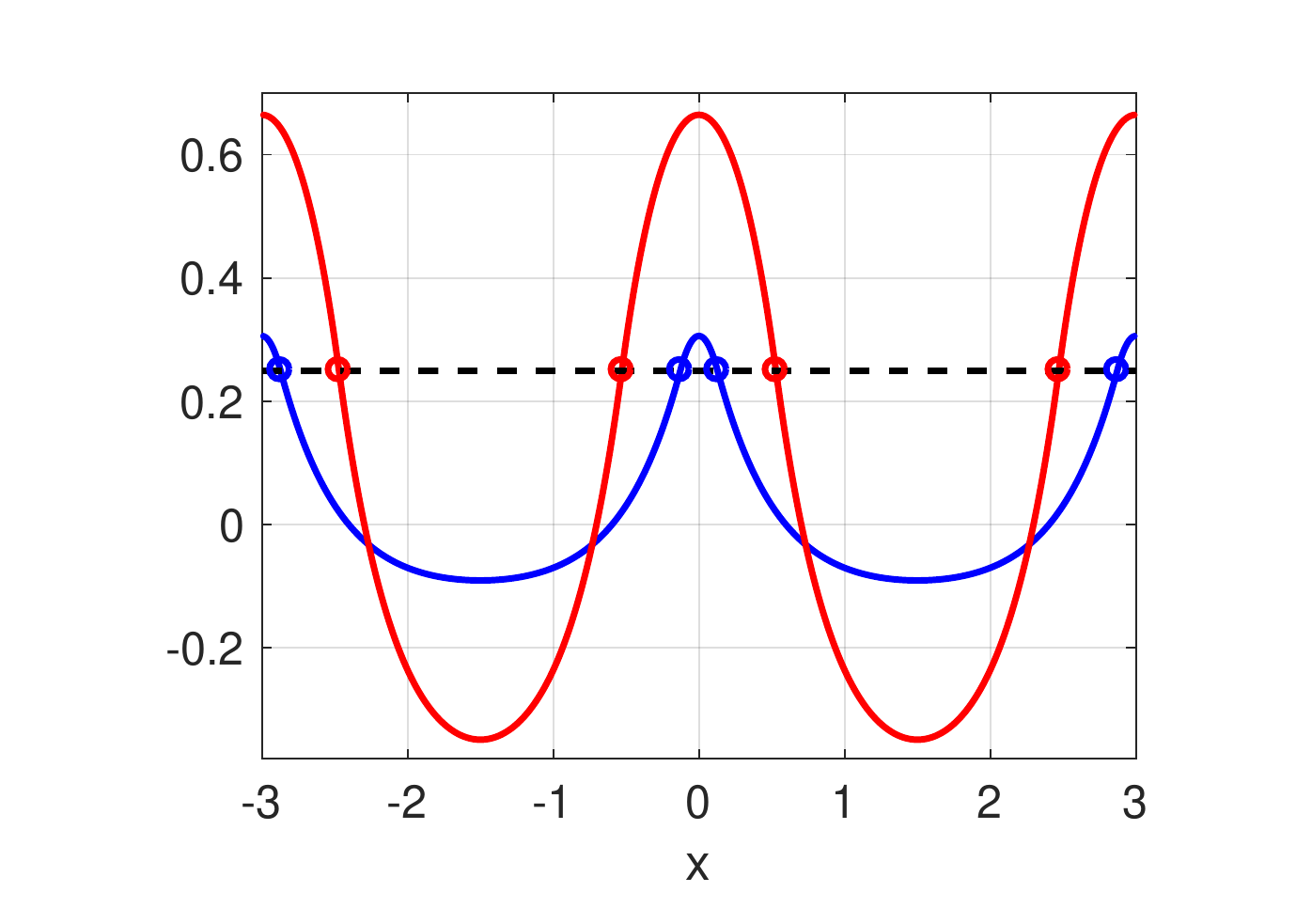}}
  \caption{(a) $1-$bump periodic solution \eqref{eq:u=int3} with $T=2.116$. The point of tangency corresponds to $a=0.2352$ and $h=0.25$. (b) $1-$bump periodic solutions \eqref{eq:u=int3} with $T=3$. The intersection points correspond to $a_1=0.1272$, $a_2=0.5288$ and $h=0.25$. (All the approximated values are rounded up to 4 decimals.)}\label{Fig:X(b)}
\end{figure}


\section{Stability of 1-bump periodic solutions}\label{Sec:Stability}

In this section we study linear stability of regular 1-bump periodic solutions.
We first obtain the Fr\'{e}chet derivative of the Hammerstein operator defined in \eqref{eq:u=Hu} and then study its spectrum.

%
\begin{lemma}\label{lemma:H'(up)}
  Let $h, T>0$ be fixed and $u_p$ be a 1-bump periodic solution of \eqref{Amari_model}. The Fr\'{e}chet derivative of the operator $\cH:C_b^{0,1}(\R)\to C_b^{0,1}(\R)$ at $u_p$ exists and is given as
  $$(\cH'(u_p)v)(x)=\frac{1}{|u'_p(a)|}\sum_{k \in \Z}(\omega(x+a-kT)v(-a+kT)+\omega(x-a-kT)v(a+kT)).$$
\end{lemma}
\begin{proof}
Due to Lemma \ref{lemma:u_p+v} and periodicity of $u_p$ the proof in \cite{oleynik2016spatially} for bumps can be easily adopted here.
\end{proof}
We would like to emphasize that the regularity condition on $u_p$, that is $|u'_p(a)|>0$,  is necessary in order for the Fr\'{e}chet derivative to exists.

Next we show how the spectrum of the operator $\cH'(u_p)$ relates to the spectrum of a Laurent block operator, or in some literature, bi-infinite block Toeplitz operator, see e.g. \cite{GohbergGoldbergKaashoek1993Vol2} and \cite{frazho2009operator,van2002spectral}.

Let $\ell_p^{m}(\Z)$ be a Banach space of sequences with entries from $\R^{m}$, see Section \ref{Sec:Notations}.

The block Laurent operator $L: \ell_p^{m}(\Z) \to \ell_p^{m}(\Z)$ can be represented as an bi-infinite matrix with constant diagonal elements, that is, $L=(A_{i-j})_{i,j\in\Z}$ giving

\begin{equation}\label{eq:Laurent op}
L=
\begin{pmatrix}
\ddots&&&&\\
&A_{0}&A_{-1}&A_{-2}&\\
&A_{1}&A_0&A_{-1}&\\
&A_{2}&A_{1}&A_0&\\
&&&&\ddots
\end{pmatrix},
\quad A_k \in \R^{m\times m}.
\end{equation}

The representation \eqref{eq:Laurent op} means that the action of $L$ is given by
\[
\left(L \left(x_n\right)_{n\in \Z}\right)=\left(y_n\right)_{n\in \Z}, \quad y_j=\sum_{i} A_{i-j}x_j.
\]

For $p=1,\infty$ we have
\begin{equation}\label{eq:T-norm}
\|L\|_{op}= \sum\limits_{k\in \Z}\|A_k\|_{op}.
\end{equation}

\begin{theorem}\label{th:AnjaVadim}
  The nonzero spectrum of the operator $\cH'(u_p)$ agrees with that of the Laurent block operator $L:\ell^{2}_{\infty}(\Z) \to \ell^{2}_{\infty}(\Z)$ defined by
 \begin{equation}\label{Eq:LaurentOp}
 A_k=\frac{1}{|u'_p(a)|} \begin{pmatrix}
                                  \omega(kT) & \omega(-2a+kT) \\
                                  \omega(2a+kT) & \omega(kT)
                                \end{pmatrix}.
\end{equation}
Moreover, any eigenfunction  $v(x)$ of $\cH'(u_p)$ (if exists) corresponds to the eigenfunction $\mathrm{v}=(\mathrm{v}_k)_{k\in\Z}$ of $L$ where
\[
\mathrm{v}_k=(v(-a+kT), v(a+kT))^T, \quad k\in \Z,
\]
and for a given eigenfunction $\mathrm{v}$ of $L$ that corresponds to a non-zero eigenvalue, we can calculate the eigenfunction of $\cH'(u_p)$ as
\[
v(x)=\frac{1}{\lambda}\frac{1}{|u'_p(a)|}\sum_{k \in \Z}(\omega(x+a-kT)\mathrm{v}_k^{(1)}+\omega(x-a-kT)\mathrm{v}_k^{(2)}.
\]
\end{theorem}
\begin{proof}

First of all we observe that $L$ is a bounded operator on $\ell^2_\infty(\Z)$ since
\[
\|L\|_{op}=\frac{1}{|u'_p(a)|} \sum\limits_{k\in \Z} \left(|\omega(kT)|+\max\{|\omega(\pm 2a+kT)|\}\right)<\infty
\]
due to Assumption \ref{as:A}.

A number $\lambda \in \C$ is in the resolvent set of the operator $\cH'(u_p)$ if and only if the equation
$$\cH'(u_p)\xi-\lambda \xi=w$$
has a solution $\xi$ for any $w$, where $\xi$ and $w$ belong to the complexified $C_b^{0,1}(\R)$.

Thus, if $\lambda\in \C$ is in the resolvent set of the operator $\cH'(u_p)$, then for any $k\in \Z$ the system of equations
$$(\cH'(u_p)\xi)(a+kT)-\lambda\xi(a+kT)\xi=w(a+kT),$$
$$(\cH'(u_p)\xi)(-a+kT)-\lambda\xi(-a+kT)\xi=w(-a+kT)$$
possesses a solution. Hence, $\lambda$ is in the resolvent set of the operator $L$.

Conversely, assume that $\lambda \ne 0$ is in the resolvent set of the operator $L$. Then for any arbitrary $w$ the values $\xi(a+kT)$ and $\xi(-a+kT)$ of the solution to $\cH'(u_p)\xi-\lambda\xi=w$ are determined. For arbitrary $x \in \R$ we set
$$\xi(x)=\frac{1}{\lambda}((\cH'(u_p)\xi)(x)-w(x)).$$
It is straightforward to verify that $\xi \in C_b^{0,1}$ and solves $\cH'(u_p)\xi-\lambda \xi=w$.
We have shown that the resolvent sets of $\cH'(u_p)$ and $L$ agree up to the point $\lambda=0$. Thus, their spectra agree up to the point $\lambda=0$ as well.
The second part of the statement follows from above.
\end{proof}

 The reader can find more information about Laurent operators and their properties in \cite{GohbergGoldbergKaashoek1993Vol2} and more recent studies \cite{van2002spectral,frazho2009operator}. The results concerning in particular the spectrum of Laurent operators can be found in \cite{reichel1992eigenvalues}. Finally, as the spectrum of Laurent operator on $\ell_2^m(\Z)$ is given by the spectrum of the corresponding matrix valued multiplication operator we refer to \cite{Denk2006spect-692} where the spectrum of the latter operator is studied.
For the original paper on the Toeplitz and Laurent operators see \cite{toeplitz1911theorie}.
Since the eigenvalue $0$ does not have any impact on the stability of $u_p,$ we now turn to the study of the Laurent operator in \eqref{eq:Laurent op} with elements as in \eqref{Eq:LaurentOp}.

As $(A_k)_{k\in \Z} \in \ell^{2\times 2}_1(\Z)$ we can define a matrix function $\Phi: S^1 \to \R^{2\times 2}$ as
\begin{equation}\label{eq:symbol}
\Phi(z)=\sum_{k\in \Z} A_k z^k, \quad z\in S^1,
\end{equation}
where $S^1$ is the unit circle.
The power series is uniformly convergent and thus the function $\Phi$  is continuous on $S^1$. The function $\Phi$ is called a symbol or a defining function of $L.$ It is easily observed that $\Phi$ belongs to the Weiner algebra of all periodic functions with absolutely summable sequence of Fourier coefficients, that is $\Phi \in \cW^{2\times2}(S^1).$ Via the Fourier transform the Banach algebra of all block Laurent operators on $\ell^2_\infty(\Z)$ is isomorphic to $\cW^{2\times 2}(S^1)$.

We prove the following important result.

\begin{theorem} \label{th:main:spectra}
\begin{itemize}
\item[(i)]The spectrum of the block Laurent operator $L: \ell^m_\infty(\Z)\to \ell^m_\infty(\Z) $ is given as
\begin{equation} \label{eq:sp}
\sigma(L)= \bigcup \limits_{z\in S^1} \sigma(\Phi(z))
\end{equation}
where $\Phi(z)$ is the symbol \eqref{eq:symbol} of $L.$

\item[(ii)] The spectrum $\sigma(L)$ is pointwise, and the eigenvectors $\mathrm{v}_\lambda=\left(v_k(\lambda)\right)_{k\in \Z}$ of $L$ can be calculated as
\begin{equation}\label{eq:z*w}
\mathrm{v}^{(k)}_\lambda=\bar{z}^k w(z_\lambda)
\end{equation}
where $z_\lambda\in S^1$ is such that $\lambda \in \sigma(\Phi(z_\lambda)),$ and $w(z_\lambda)$ is the corresponding eigenvalue of the matrix $\Phi(z_\lambda).$
\end{itemize}
\end{theorem}
\begin{proof}
To prove the first statement we recall that invertibility (and Fredholmness) of operators on the Wiener algebra is independent on underlying space, see \cite{Lindner2008,seidel2014fredholm} and  references therein. That is, the spectrum of $L: \ell_p^m(\Z)\to \ell_p^m(\Z)$, does not depend on $1\leq p\leq \infty,$ and is given by all the values $\lambda \in \C$ such that $\det(\Phi(z)-\lambda I)=0$
for some $z\in S^1,$ see \cite{GohbergGoldbergKaashoek1993Vol2,reichel1992eigenvalues} and \cite{Denk2006spect-692}.\\

To prove the second statement let $\lambda \in \sigma(L)$. From \eqref{eq:sp} there exists $z_\lambda=\exp(i\theta_\lambda)$, $\theta_\lambda \in [0,2\pi)$ such that
\[
\det(\Phi(z_\lambda)-\lambda I)=0.
\]

Thus, there exists an eigenvector $w(z_\lambda) \in \C^{m}$ such that
\[
\Phi(z_\lambda)w(z_\lambda)=\lambda w(z_\lambda).
\]

Let us define $v \in \ell^m_\infty(\Z)$ as follows
\[
v=\{v_k\}_{k\in \Z}, \quad v_k=e^{-\i k\theta_\lambda} w(z_\lambda).
\]

It is easy to check that $v\in \ell^m_\infty(\Z)$ and is the eigenfunction of the Laurent operator $L$ corresponding to $\lambda.$
Indeed, for the $n$th row we have
\[
\begin{split}
(L v)_n&=\sum\limits_{k\in \Z} A_{k-n} e^{\i k\theta_\lambda}w(z_\lambda)=
\sum\limits_{l\in \Z} A_l e^{\i(n+l)\theta_\lambda}w(z_\lambda)=\\
&=e^{\i n\theta_\lambda}\sum\limits_{l\in \Z} A_l e^{\i l\theta_\lambda}w(z_\lambda)=
e^{\i n\theta_\lambda}´\Phi(z_\lambda)=e^{\i n\theta_\lambda} \lambda w(z_\lambda)=\\
&=\lambda v_n.
\end{split}
\]
\end{proof}

Next, we describe some properties of the symbol $\Phi$ that corresponds to the Laurent operator \eqref{Eq:LaurentOp}.

\begin{lemma}\label{lemma:phi}
The matrix $\Phi(z)$ in \eqref{eq:symbol} with $A_k$ given by \eqref{Eq:LaurentOp} is  self-adjoint and $\overline{\Phi(z)}=\Phi(\overline{z}),$ $z\in S^1.$
\end{lemma}
\begin{proof}
The second property follows directly from \eqref{eq:symbol} and $\omega(x)$ being real.
To show that $\Phi(z)$ is self-adjoint let $\theta \in [0,2\pi).$ Then we have
\begin{equation*}
\begin{split}
\overline{\Phi(z)^T}=&
\sum_{k\in Z}
\begin{pmatrix}
\omega(kT) &\omega(2a+kT)\\
\omega(-2a+kT)&\omega(2kT)
\end{pmatrix} e^{-\i k\theta}=\\
&\sum_{m\in Z}
\begin{pmatrix}
\omega(-mT) &\omega(2a-mT)\\
\omega(-2a-mT)&\omega(-mT)
\end{pmatrix}
 e^{\i m \theta}=\Phi(z)
\end{split}
\end{equation*}
as $\omega(x)$ is symmetric, see Assumption \ref{as:A}(i).
\end{proof}

From Lemma \ref{lemma:phi} and Theorem  \ref{th:main:spectra}(i) the spectrum of $L$, and consequently of $\cH'(u_p)$, is real and
\begin{equation}
\label{eq:spec}
\sigma(L)=\bigcup_{z\in S^1} \left(\lambda_{1,2}(z)\right)
\end{equation}
where
\begin{equation}\label{eq:lambdas}
\lambda_1(z)=\Phi_{11}(z)- |\Phi_{12}(z)| \mbox{ and } \lambda_2(z)=\Phi_{11}(z)+ |\Phi_{12}(z)|
\end{equation}
 and $\Phi_{ij}(z)$ are the entries of the symbol matrix $\Phi(z).$ Moreover, it is enough to consider only half of the circle, that is, $z=e^{\i\theta}$ with $\theta \in [0,\pi].$

Let now $z_\lambda=e^{\i\theta}$ in Theorem \ref{th:main:spectra}(ii)  with $\theta/(2\pi)$ being a rational number from $[0,0.5],$ i.e. $\theta/(2\pi)=p/q,$ $p\cup\{0\},$ $q\in \N$  where $p$ and $q$  are in the lowest terms. Then from \eqref{eq:z*w} the corresponding eigenvector $\mathrm{v}$ is $(1+q)$-periodic.
If $\lambda\not=0$ then from Theorem \ref{th:AnjaVadim}, the eigenfunction $v$ of $\cH'(u_p)$ is $(1+q)T$-periodic. Thus, we can calculate the spectrum even without calculating the symbol $\Phi$. We summarize it as a theorem.

\begin{theorem}\label{th:clAn}
The spectrum of the operator $L$ is given as
\[
\sigma(L)=cl\left(\bigcup \sigma\left(L(1+q)\right)\right)
\]
where $L(1+q)$, $q=1,2,...,$ are $2(1+q)\times 2(1+q)$ matrices given as
\[
L(1+q)=
\begin{pmatrix}
B_0& B_1&B_2& ...&B_{q}\\
B_{q}&B_0&B_1&...&B_{q-1}\\
\vdots&\vdots&\vdots&\vdots&\vdots\\
B_1&B_2&B_3&...&B_0
\end{pmatrix}
\]
where
\[
B_n=\frac{1}{|u'_p(a)|}
\begin{pmatrix}
\omega_p(nT;(1+q)T)& \omega_p(-2a+nT;(1+q)T)\\
\\
\omega_p(2a+nT;(1+q)T)& \omega_p(nT;(1+q)T)
\end{pmatrix}, n=0,...,q.
\]
\end{theorem}

We illustrate Theorem \ref{th:clAn} in Fig.\ref{Fig:spEx1}(b) for $\omega$ as in \eqref{eq:exp}.

When $q=0$ we readily calculate $L(1)=B_0$ where
\[B_0=\frac{1}{u'_p(a)}
\begin{pmatrix}
\omega_p(0;T)& \omega_p(2a;T)\\
\\
\omega_p(2a;T)& \omega_p(0;T)
\end{pmatrix},
\]
has the eigenvalues
\begin{equation} \label{eq:lambdas12}
\lambda_{1,2}=\frac{\omega_p(0;T)\pm \omega_p(2a;T)}{\omega_p(0;T)- \omega_p(2a;T)}
\end{equation}
or, equivalently,
$\lambda_{1}=1$ and $\lambda_2=1+2\omega_p(2a;T)/|\omega_p(0;T)- \omega_p(2a;T)|.$

These eigenvalues are similar to the ones obtained for bump solutions. Indeed, for a bump solution one can compute the corresponding eigenvalues of the Fr\'{e}chet operator (at a bump solution) as $\mu_{1}=1$ and $\mu_2=1+2\omega(2a)/|\omega(0)- \omega(2a)|,$ see e.g.\cite{KosOl2013unstable}.
The first eigenvalue $\lambda_1=1$ ($\mu_1=1$) corresponds to the translation of the solution, see \cite{KosOl2013unstable}. Thus, for the bump solutions, the sign of $\omega(2a)$ will define the linear stability.
In the case of 1-bump periodic solutions,  $\omega_p(2a;T)>0$ implies instability. If $\omega_p(2a;T)<0$ the eigenvalues of $L(2),$ then $L(3)$ and etc., must be calculated. The structure of $L(1+q)$ could be useful in exploring spectrum if the analytic expression for $\Phi$ is not available.

As we aim at studying Lyapunov stability of 1-bump periodic solutions for \eqref{Amari_model} with smooth sigmoid like function $f$ by deriving spectral asymptotic, the eigenvalue $1$ ideally must be isolated and have multiplicity one. We believe that the second condition could be satisfied under some additional assumptions on $\omega_p,$ including  $\omega_p(2a;T) \not=0$. The first condition, however, is never satisfied. Thus one must employ more detailed analysis of spectral convergence than in the case of bump solutions  \cite{OKS(Unpublished)}.
However, this is  out of the scope of this paper.

In the next section we apply the theory above to study linear stability of the 1-bump periodic solutions from Section \ref{Sec:Existence:Ex}.

\subsection{Examples} \label{Sec:Stability:Ex}

Define the auxiliary functions
\begin{equation}
\alpha(\theta;s,T)=\frac{\sinh(sT)}{\cosh(sT)-\cos(\theta)}
\end{equation}
and
\begin{equation}
\beta(\theta;s,T)=\frac{\sinh(2as)e^{-\i \theta}+\sinh(s(T-2a))}{\cosh(sT)-\cos(\theta)}.
\end{equation}

Then, for $\omega$ as in \eqref{eq:exp} we obtain
\[
\Phi(e^{\i\theta})=\dfrac{S}{|u'_p(a)|}
\begin{pmatrix}
\alpha(\theta;s)& \overline{\beta(\theta;s)}\\
\beta(\theta;s)&\alpha(\theta;s)
\end{pmatrix}.
\]

In Fig. \ref{Fig:spEx1}(a) we plot $\lambda_1(e^{\i\theta})$ and $\lambda_2(e^{\i\theta})$ as functions of $\theta,$ $\theta \in [0,2\pi)$ for $T=4$ and the parameters as in Fig.\ref{Fig:wpWp1andBump}. As $\lambda_i(z)-1<0$ the 1-bump periodic solution is linearly unstable. It can be shown that this is always the case for all admissible parameters and any $T>0.$ Indeed, for $S=0.5$ and $s=1,$ we obtain $a\to -0.5 \log(|2h-1|)$ as $T\to \infty$ and $\lambda_1 \to 1$ while $\lambda_2 \to 1/h-1>1$. We notice that these values could be obtained by passing the limit in \eqref{eq:lambdas12}. In Fig.\ref{Fig:SpU_exp} we plot the minimum and maximum of $\lambda_2(e^{\ii \theta})$ (red curves) and  $\lambda_1(e^{\ii \theta})$ (blue curves) for different $T.$ As $T\to 0,$ $\max_\theta \lambda_2(e^{\ii \theta}) \to \infty.$

\begin{figure}[H]
  \centering
\includegraphics[width=7cm]{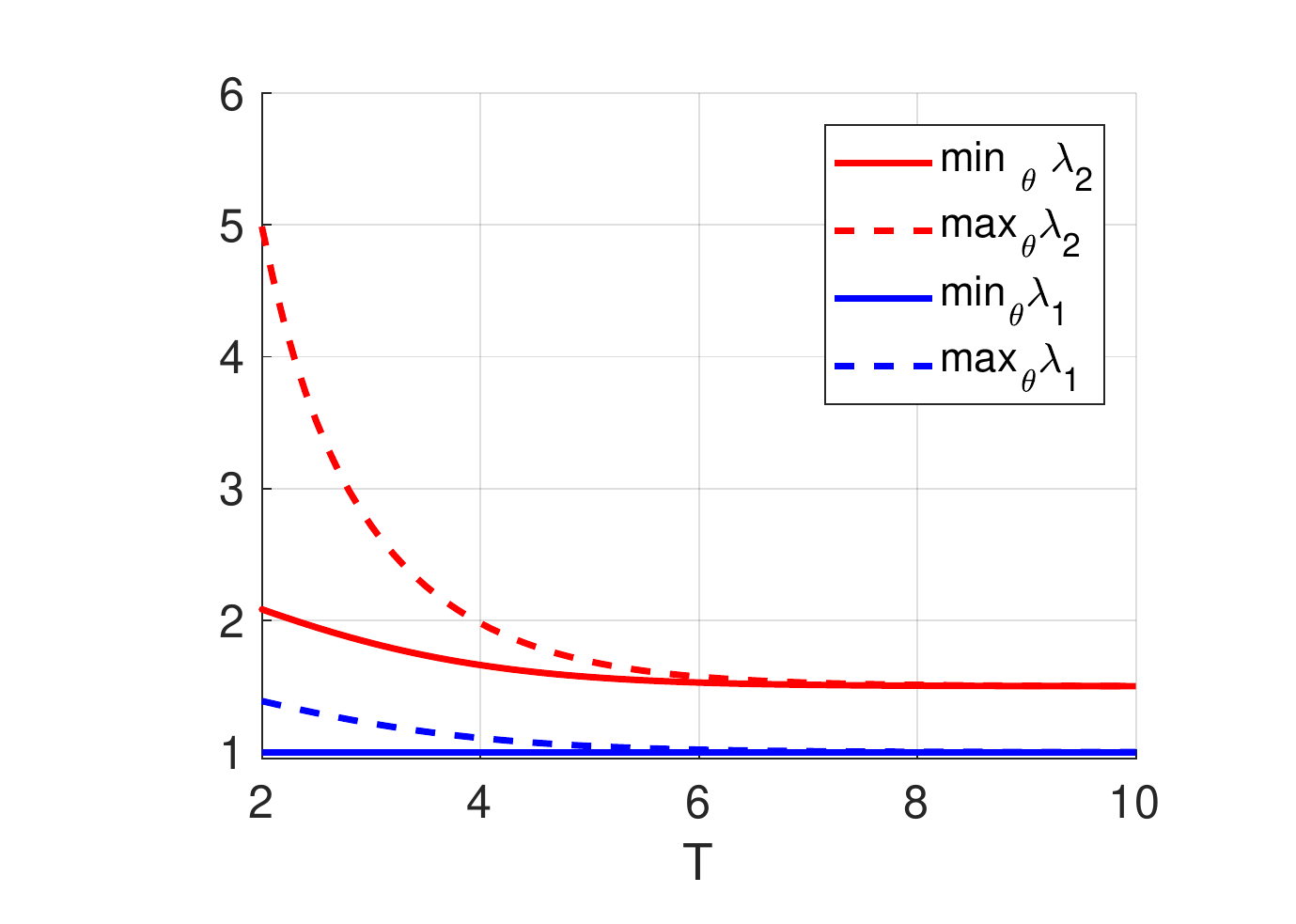}
 \caption{Bounds for $\sigma(L)$ in \eqref{eq:spec} depending on $T$ when $\omega$ is given by \eqref{eq:exp} with $S=0.5,$ $s=1,$ and $h=0.4.$  }
 \label{Fig:SpU_exp}
\end{figure}

In order to illustrate Theorem \ref{th:clAn}, we plot the eigenvalues of the matrices $L(n)$ for $n=6,$ $n=10$ and $n=50$ in Fig. \ref{Fig:spEx1}(b).


\begin{figure}[H]
  \centering
  \subfigure[]{  \includegraphics[width=7cm]{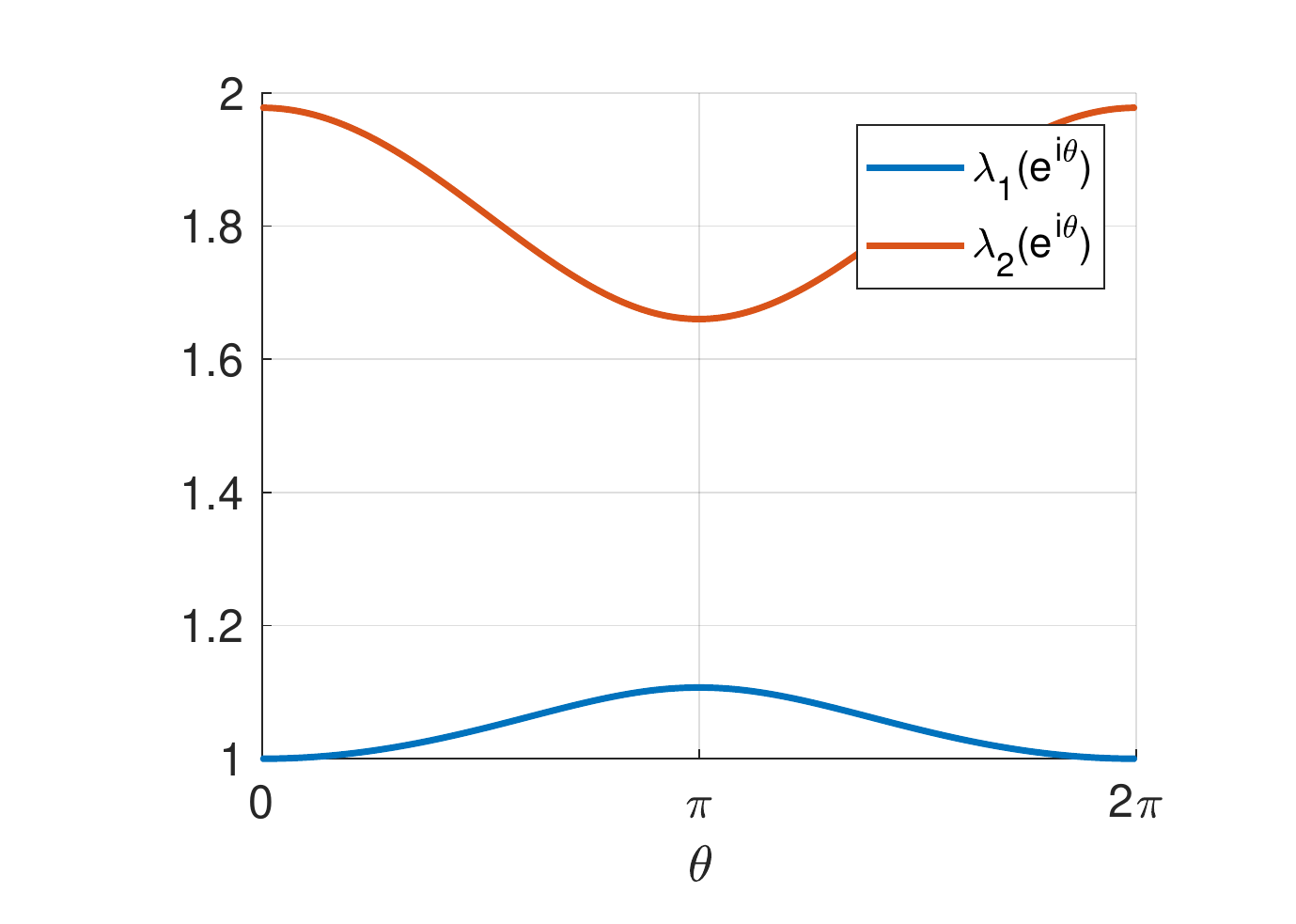}}
  \subfigure[]{\includegraphics[width=7cm]{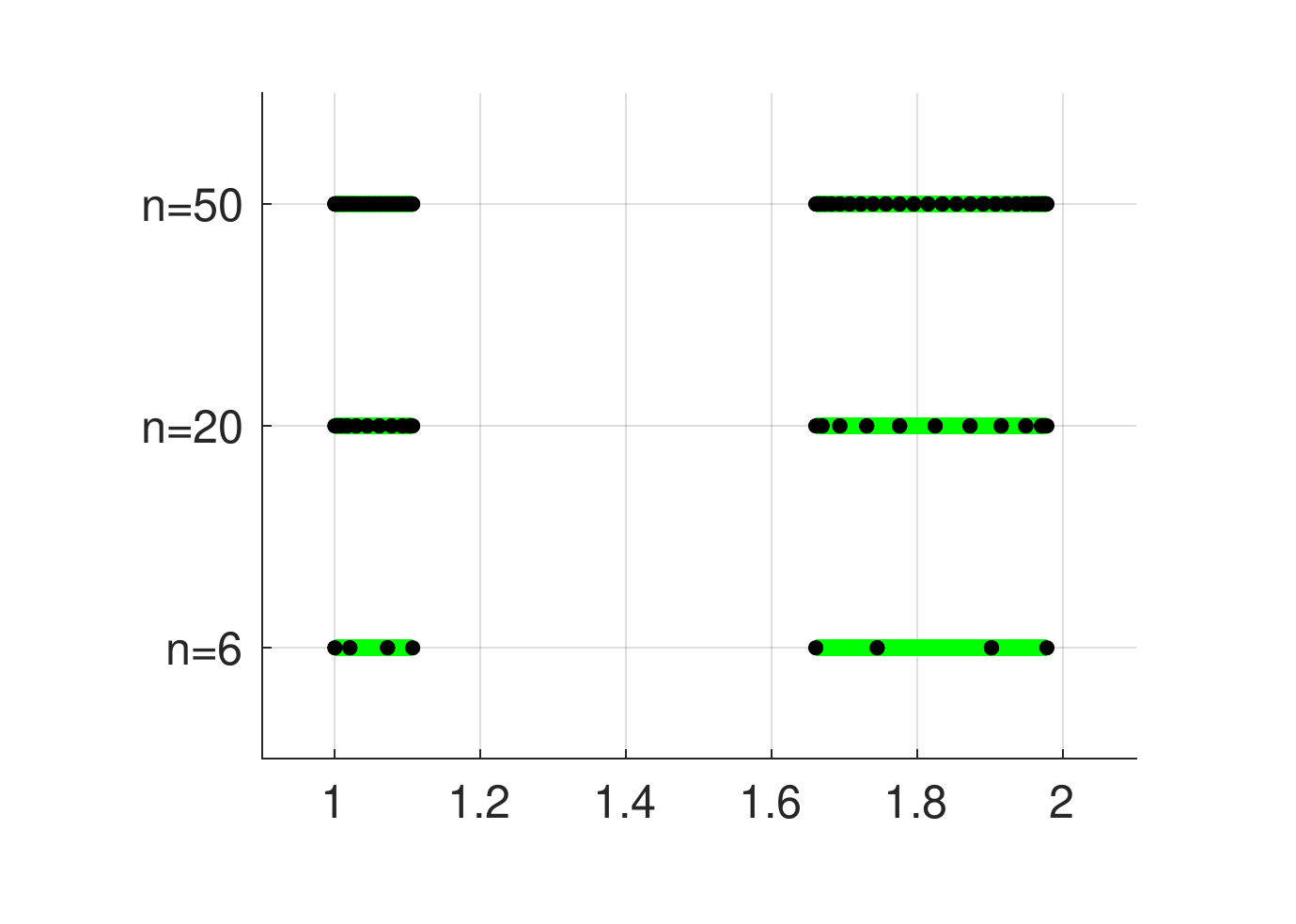}}
    \caption{(a) The eigenvalues $\lambda_{1,2}(e^{\i\theta})$ as functions of $\theta$ when $\omega$ is given by \eqref{eq:exp} with parameters $S=0.5,$ $s=1$, $h=0.4$ and $T=4.$ (b) The eigenvalues of the matrices $L(n)$ for $n=6,$ $n=20$ and $n=50$ (black dots) with the same parameters as in (a).}\label{Fig:spEx1}
\end{figure}

Let us consider $\omega$  in \eqref{eq:exp-exp}. We  readily find
\[
\Phi(e^{\i\theta})=\dfrac{1}{|u'_p(a)|}
\begin{pmatrix}
S_1\alpha(\theta,s_1,T)-S_2\alpha(\theta,s_2,T)& \overline{S_1\beta(\theta,s_1,T)-S_2\beta(\theta,s_2,T))}\\
S_1\beta(\theta,s_1,T)-S_2\beta(\theta,s_2,T)&S_1\alpha(\theta,s_1,T)-S_2\alpha(\theta,s_2,T)
\end{pmatrix}
\]
with $u'_p(a)$ given by \eqref{eq:u'(a)}.

For this case we have different cases depending on $T,$ see Table 1.
\begin{table}[H]
\begin{center}
\begin{tabular}{c c c c c c c c c c c c c c}
\hline\hline
Parameters & Number of solutions & Stability \\\hline
$0<T<T_1$ & One solution $u_{p,1}$& Unstable \\\hline
$T=T_1$ & Two solutions $u_{p,1}$ and $u_{p,cr}$ &Unstable \\\hline
$T\in(T_1, T_2)$ & Tree solutions $u_{p,i}$  & Unstable \\\hline
$T\geq T_2$ & Tree solutions $u_{p,i}$  & $u_{p,1}, u_{p,3}$ are unstable, \\ &  & $u_{p,2}$ is stable\\\hline\hline
\end{tabular}
\caption{$u_{p,k}=u_p(x; a_i)$, $i=1, 2, 3$ are $1$-bump periodic solutions for $a_1\le\ a_2\le a_3$. For parameters $S_1=4$, $s_1=2$, $S_2=1.5$, $s_2=1$ and $h=0.4$ we have $T_1=2.4997$, $T_2=3.3320$ and examples of $u_{p,\cdot}$ given in Fig. \ref{Fig:expexpBumps12}-Fig. \ref{Fig:expexpBumps33}.}
\label{table:1}
\end{center}
\end{table}


The solution $u_{p,1}$ is always unstable, see Table \ref{table:1}. Similarly to the previous examples, we plot spectral bounds in Fig. \ref{Fig:SpU1}. In Fig. \ref{Fig:SpU1}(b) we plot the boundaries of $\lambda_1(z)$ to illustrate that at $T=T_1$ the eigenvalue becomes less than 1, which in this case does not effect the stability of the solution.
\begin{figure}[H]
  \centering
  \subfigure[]{\includegraphics[width=7cm]{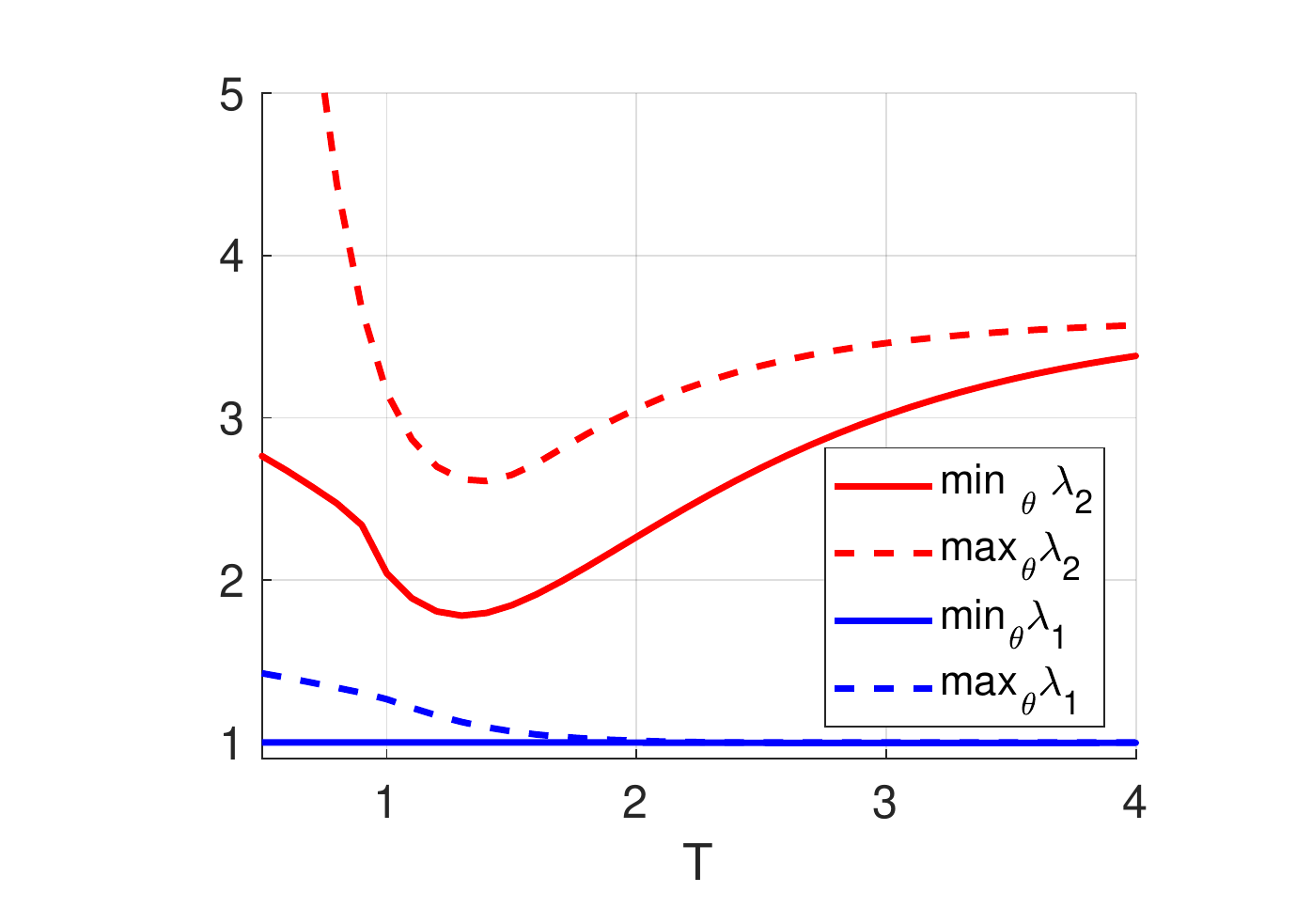}}
\subfigure[]{\includegraphics[width=7cm]{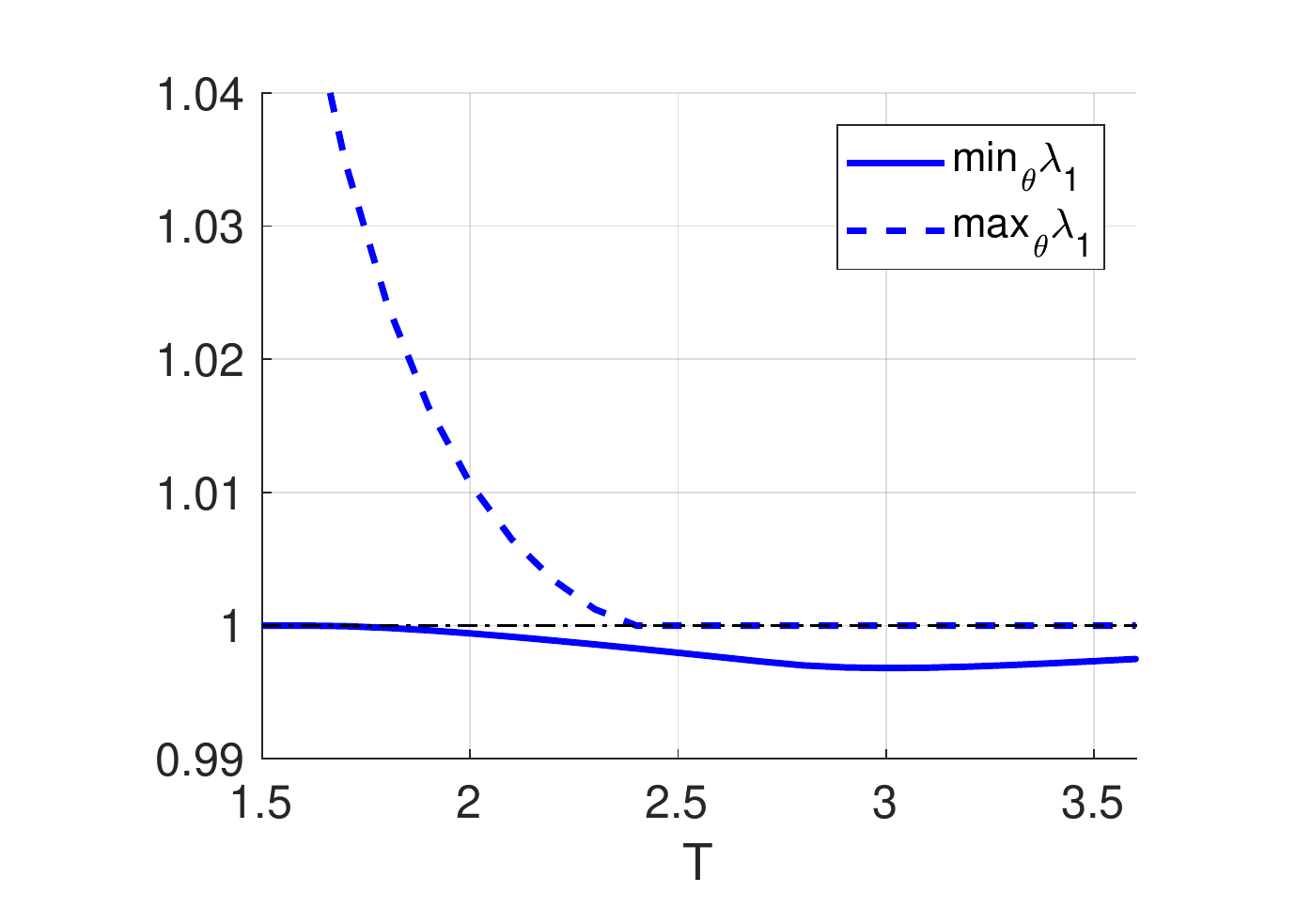}}
 \caption{Bounds for $\sigma(L)$ when $u_p=u_{p,1},$ depending on $T$. Here  $\omega$ is given by \eqref{eq:exp-exp} with  $S_1=4$, $s_1=2$, $S_2=1.5$, $s_2=1$ and $h=0.4,$ see Table \ref{table:1}. }
 \label{Fig:SpU1}
\end{figure}


The period $T=T_1$ corresponds to the critical situation where the new linearly unstable solution $u_{p,cr}$ appears, and splits into two unstable solutions $u_{p,2}$ and $u_{p,3}$ for $T=T_1+\epsilon,$ $\epsilon>0.$
The spectrum of $L$ in this case has no spectral gap, see Fig. \ref{Fig:lambda12Ucr}.

\begin{figure}[H]
  \centering
\includegraphics[width=7cm]{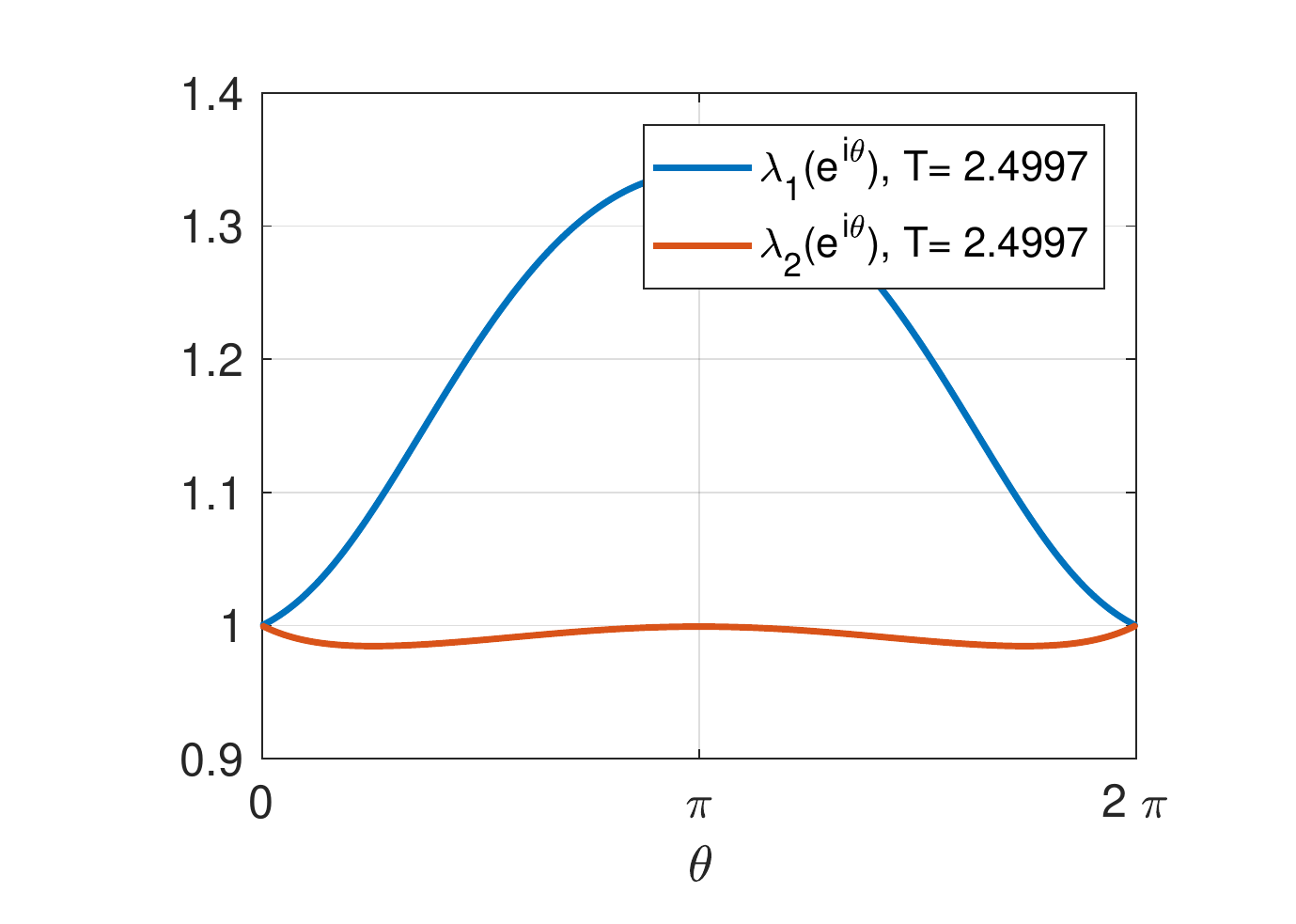}
 \caption{The eigenvalues $\lambda_1(e^{\i \theta})$ and $\lambda_1(e^{\i \theta})$ when $u_p=u_{p,cr}$. Here $\omega$ is given as in \eqref{eq:exp-exp} with $S_1=4, s_1=2$, $S_2=1.5$, $s_2=1,$ $h=0.4,$ the critical period value  $T=T_1=2.4997$ giving $\sigma(L)=[9.8460,1.3403]$ (all the approximated values are rounded up to 4 decimals). }
 \label{Fig:lambda12Ucr}
\end{figure}

For the solution $u_{p,2}$ we plot the bifurcation diagram in Fig. \ref{Fig:SpU2}.
The red curves corresponds to the minimum and maximum of $\lambda_2$ and blue to the minimum and maximum values of $\lambda_1$ for different $T.$ From \eqref{eq:lambdas} the spectrum of $\cH'(u_{p,2})$ lies in between of red and blue curves.

\begin{figure}[H]
  \centering
\includegraphics[width=7cm]{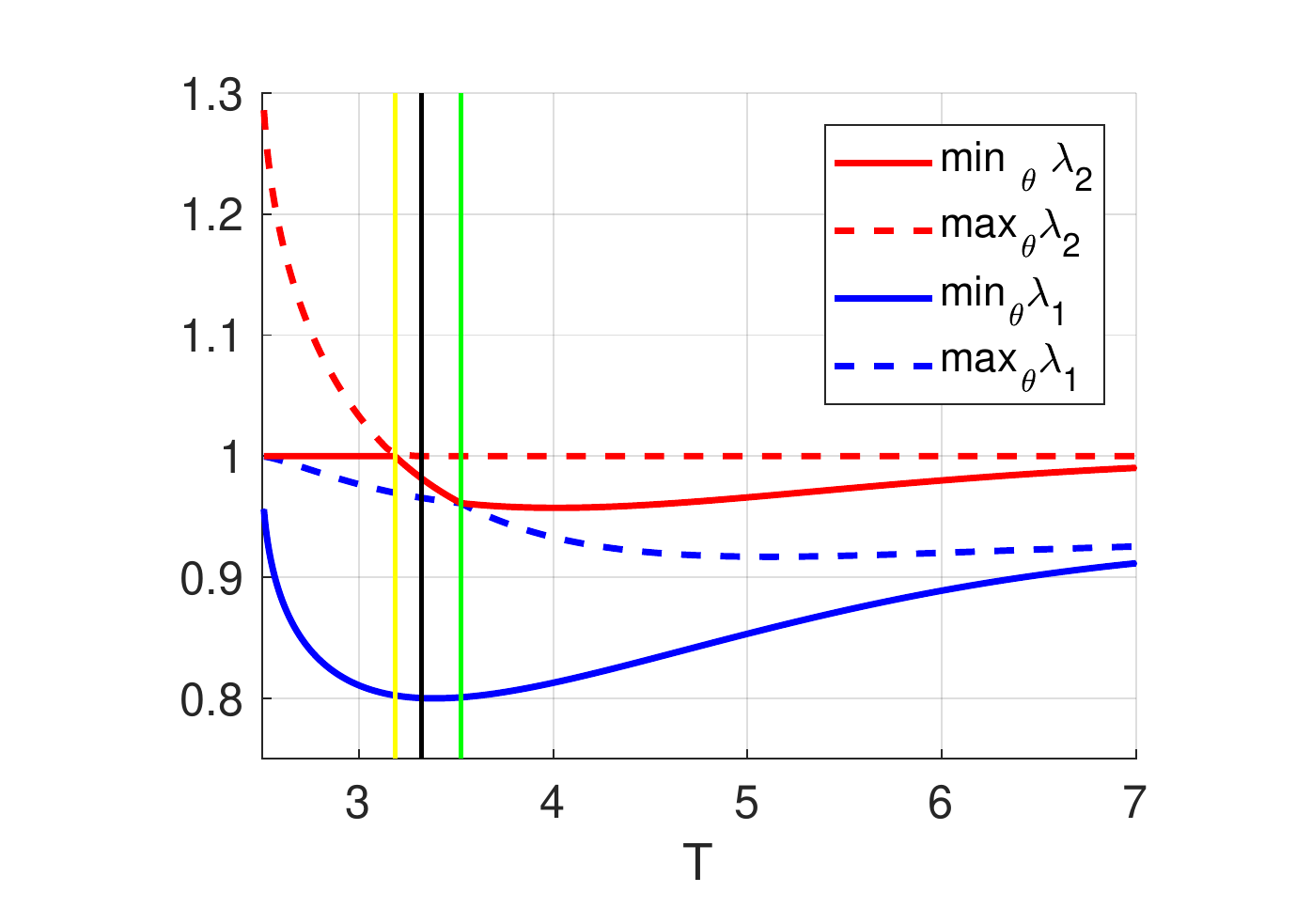}
 \caption{Bounds for $\sigma(L),$ when $u_p=u_{p,2}$, depending on $T$. The marked values corresponds to  $T=3.1849$ (yellow), $T=3.3320$ (black) and $T=3.5243$ (green) (all the values are rounded up to 4 decimals).  Here $\omega$ is given as in \eqref{eq:exp-exp} with $S_1=4, s_1=2$, $S_2=1.5$, $s_2=1,$ $h=0.4,$ see Table \ref{table:1}. }
 \label{Fig:SpU2}
\end{figure}

 The point $T=3.1849$ in Fig.\ref{Fig:SpU2} seemingly appears as a bifurcation point. This is however not the case and $T=3.1849$ only corresponds to the situation when minimum of $\lambda_2(e^{\ii \theta})$ becomes negative. In order to clarify this point we plot $\lambda_{2}(e^{\ii \theta})$ for $T=3.18,$ $T=3.1849$ in  Fig. \ref{Fig:lambdasU2}(a).
 We also plot $\lambda_{2}(e^{\ii \theta})$ for $T=3.25$ and the bifurcation point $T=T_2=3.3320$ in  Fig. \ref{Fig:lambdasU2}(a).
 For $T=3.5243$ the spectrum is again a connected set $\sigma(L)=[0.8007,1]$, see Fig.\ref{Fig:lambdasU2}(b).

\begin{figure}[H]
  \centering
  \subfigure[]{  \includegraphics[width=7cm]{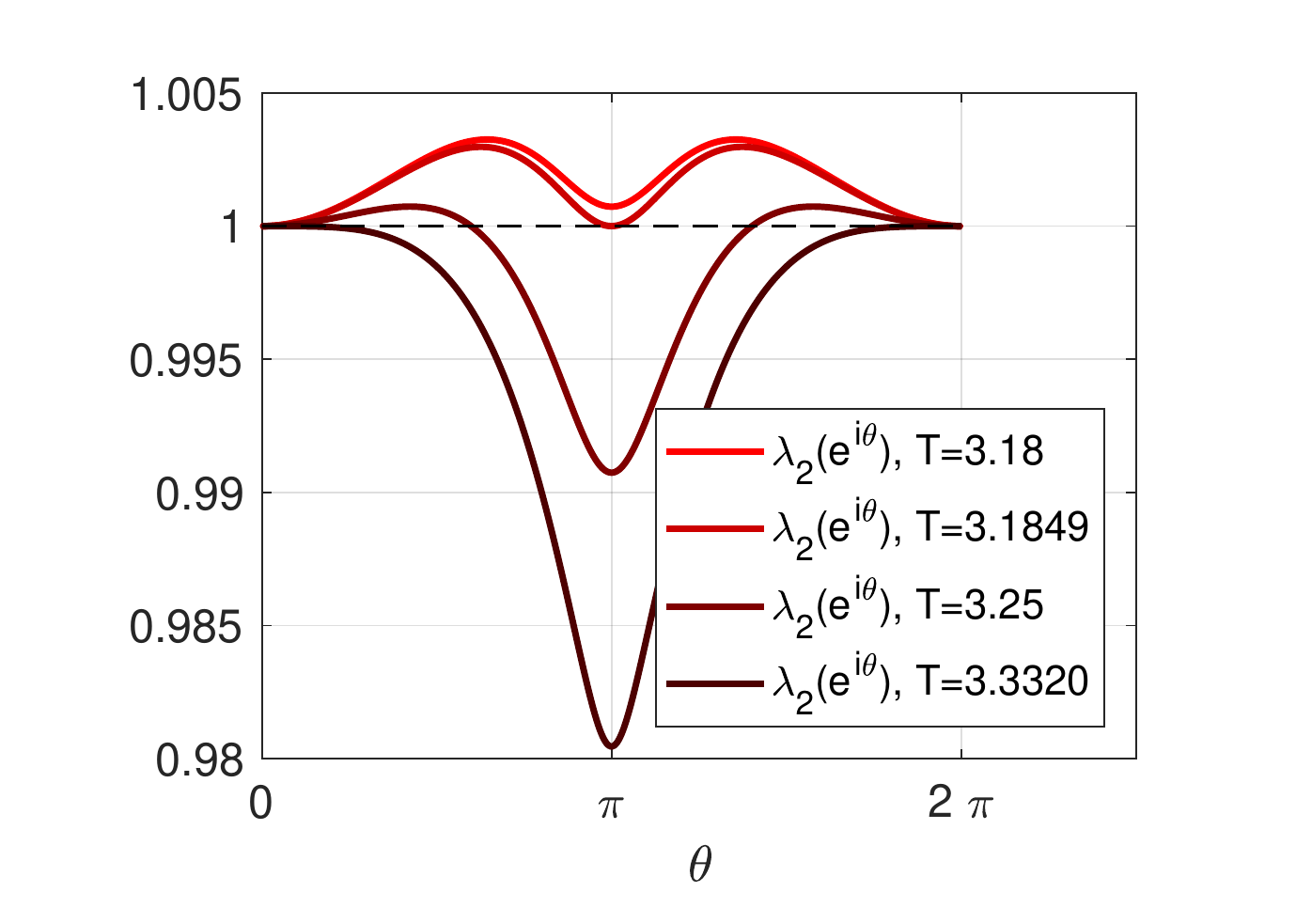}}
  \subfigure[]{\includegraphics[width=7cm]{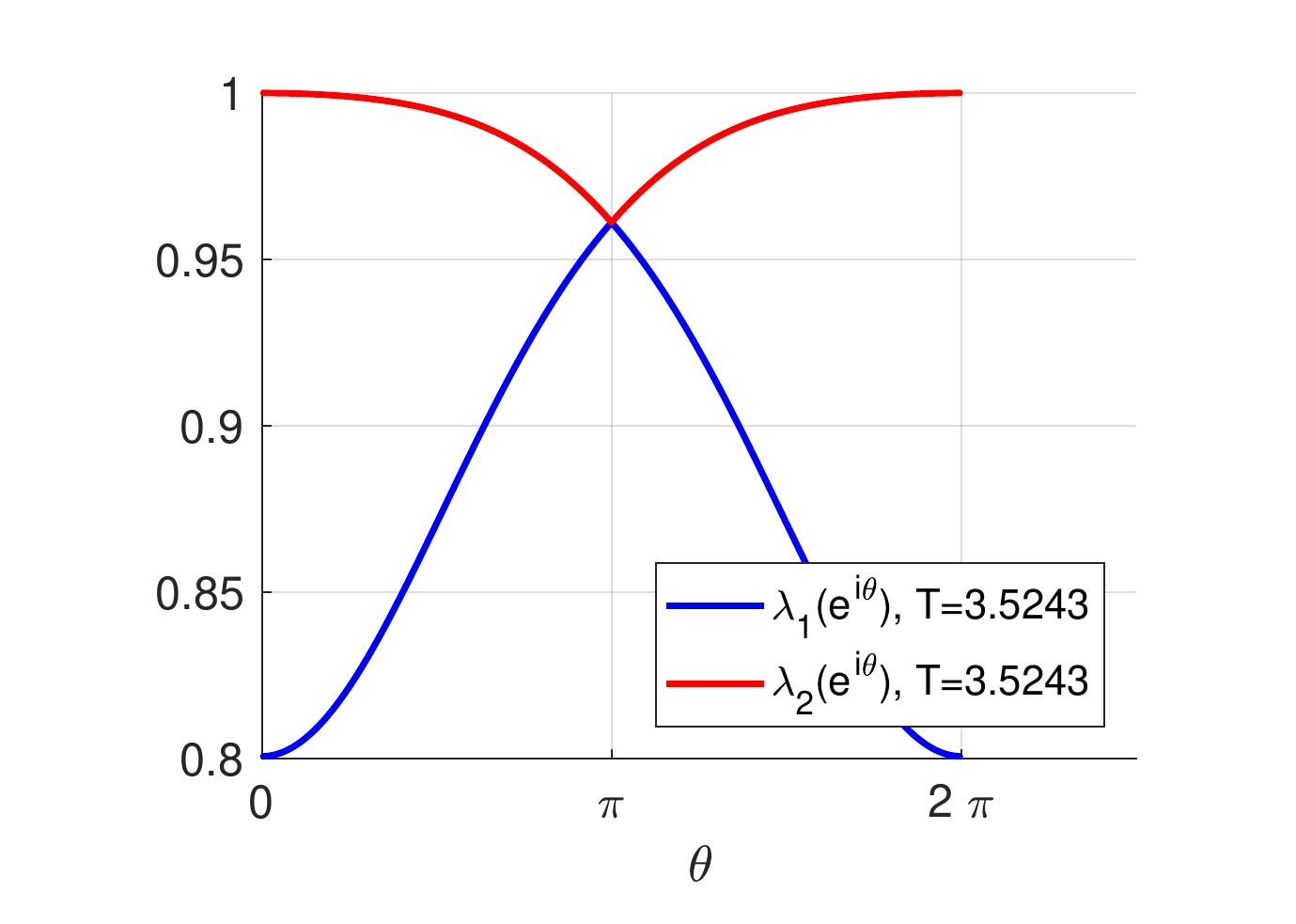}}
 \caption{The eigenvalue $\lambda_2(e^{\ii \theta})$ in (a) and $\lambda_{1,2} (e^{\ii \theta})$ in (b) when $u_p=u_{p,2},$ see Table \ref{table:1} for different $T$.  Here $\omega$ is given as in \eqref{eq:exp-exp} with $S_1=4, s_1=2$, $S_2=1.5$, $s_2=1,$ $h=0.4.$}
 \label{Fig:lambdasU2}
\end{figure}

Similarly, we plot the spectral bounds for $u_{p,3}$ in Fig.\ref{Fig:SpU3}.
\begin{figure}[H]
  \centering
\includegraphics[width=7cm]{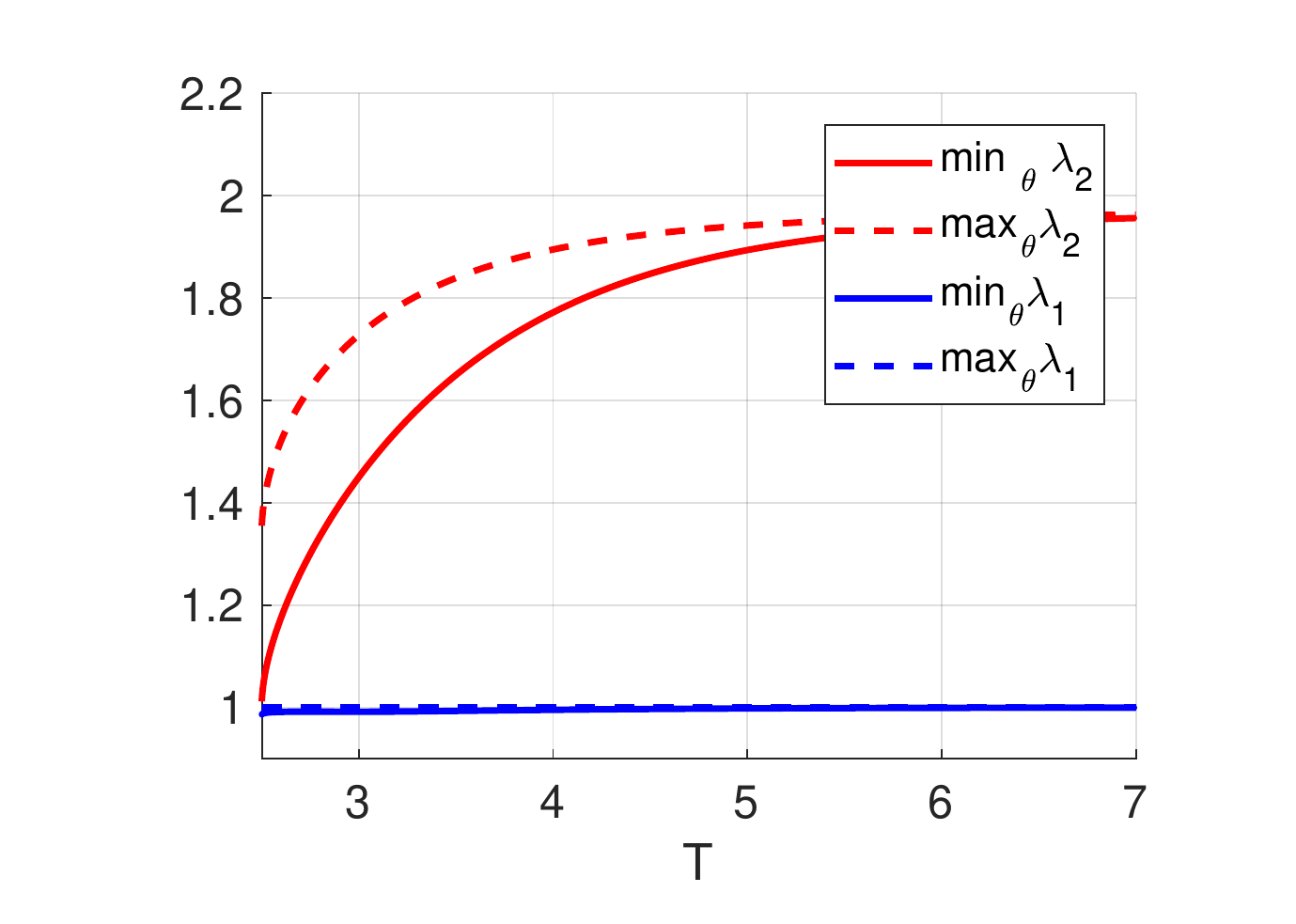}
 \caption{Bounds for $\sigma(L)$ when $u_p=u_{p,3},$ depending on $T$, see Table \ref{table:1}. Here $\omega$ is given as in \eqref{eq:exp-exp} with $S_1=4, s_1=2$, $S_2=1.5$, $s_2=1,$ $h=0.4.$  }
 \label{Fig:SpU3}
\end{figure}

As $T\to \infty$ the limiting values could be calculated from \eqref{eq:lambdas12} once the limiting expression for $a(T)$ is obtained. The calculations however are cumbersome and we omit them here. The numerical calculations however indicate, as illustrated in Fig.\ref{Fig:SpU1}, \ref{Fig:SpU2} and \ref{Fig:SpU3}, that there are no stability changes for larger period $T.$

We plot examples of $\lambda_1(e^{\i\theta})$ and $\lambda_2(e^{\i\theta})$ as functions of $\theta$ for $T=1.5$,$T=3.2$ and $T=3.5$ for every solution $u_{p,i},$ $i=1,2,3,$ in Fig. \ref{spec2}-\ref{spec2c}.

\begin{figure}[H]
  \centering
  \includegraphics[width=7cm]{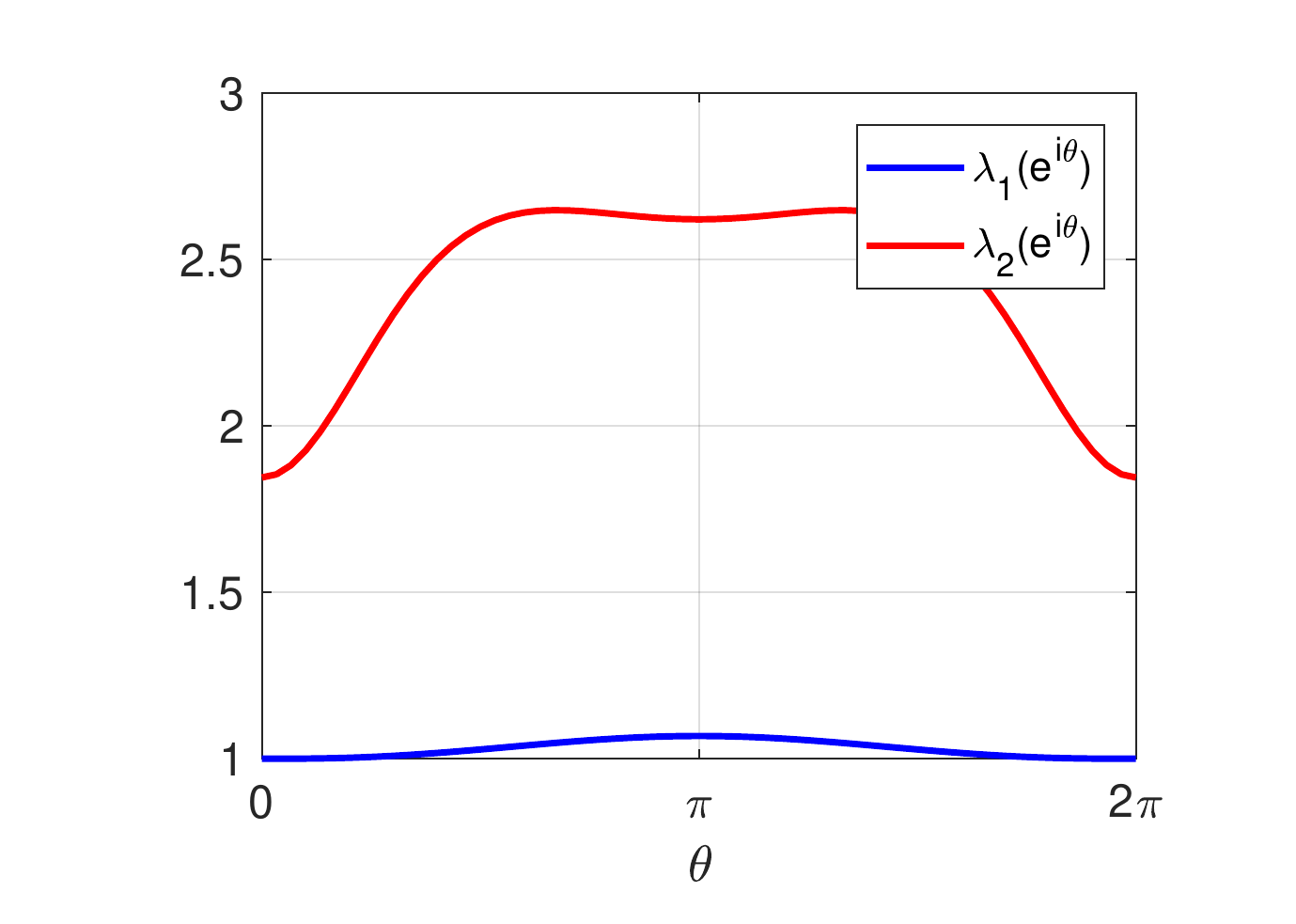}
  \caption{The eigenvalues $\lambda_{1,2}(e^{\i\theta})$ when $u_p=u_{p,1}$. Here $\omega$ is given as in \eqref{eq:exp-exp} with $S_1=4, s_1=2$, $S_2=1.5$, $s_2=1,$ $h=0.4$, and $T=1.5$. The resulting spectrum $\sigma(L)=[1, 1.0684]\cup[1.8449, 2.6479]$}\label{spec2}
\end{figure}


\begin{figure}[H]
  \centering
  \subfigure[]{\includegraphics[width=4.5cm]{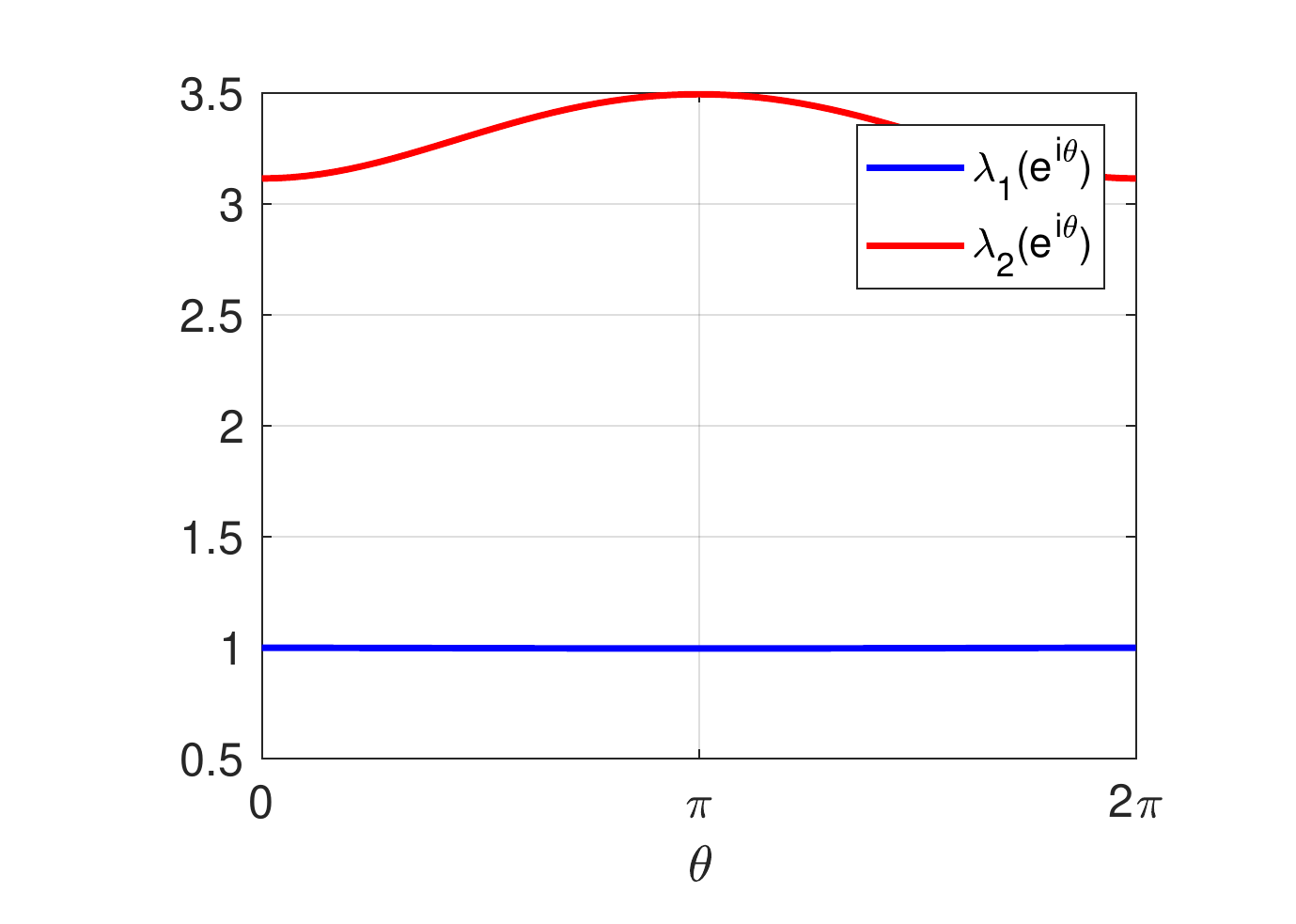}}
  \subfigure[]{\includegraphics[width=4.5cm]{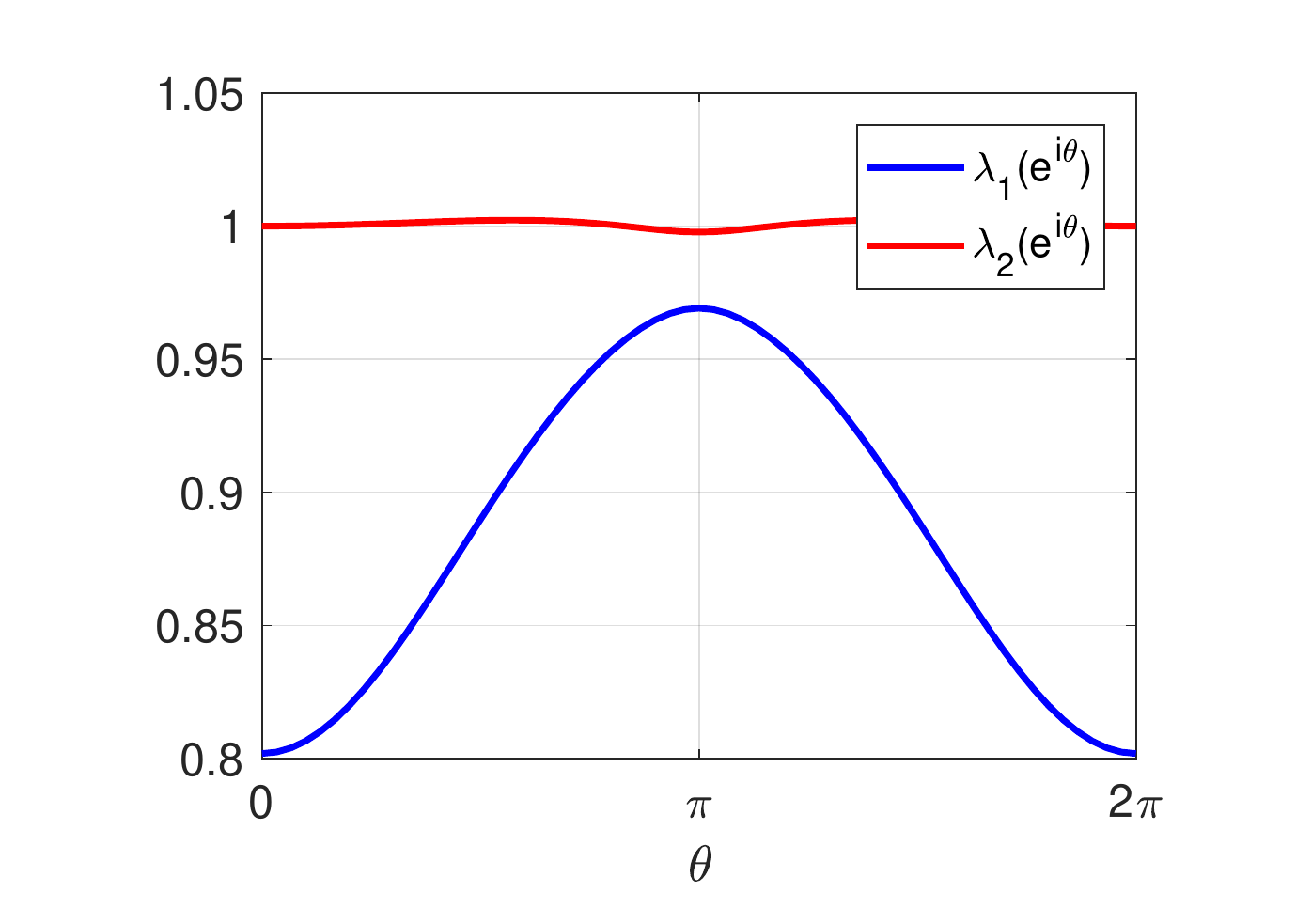}}
  \subfigure[]{\includegraphics[width=4.5cm]{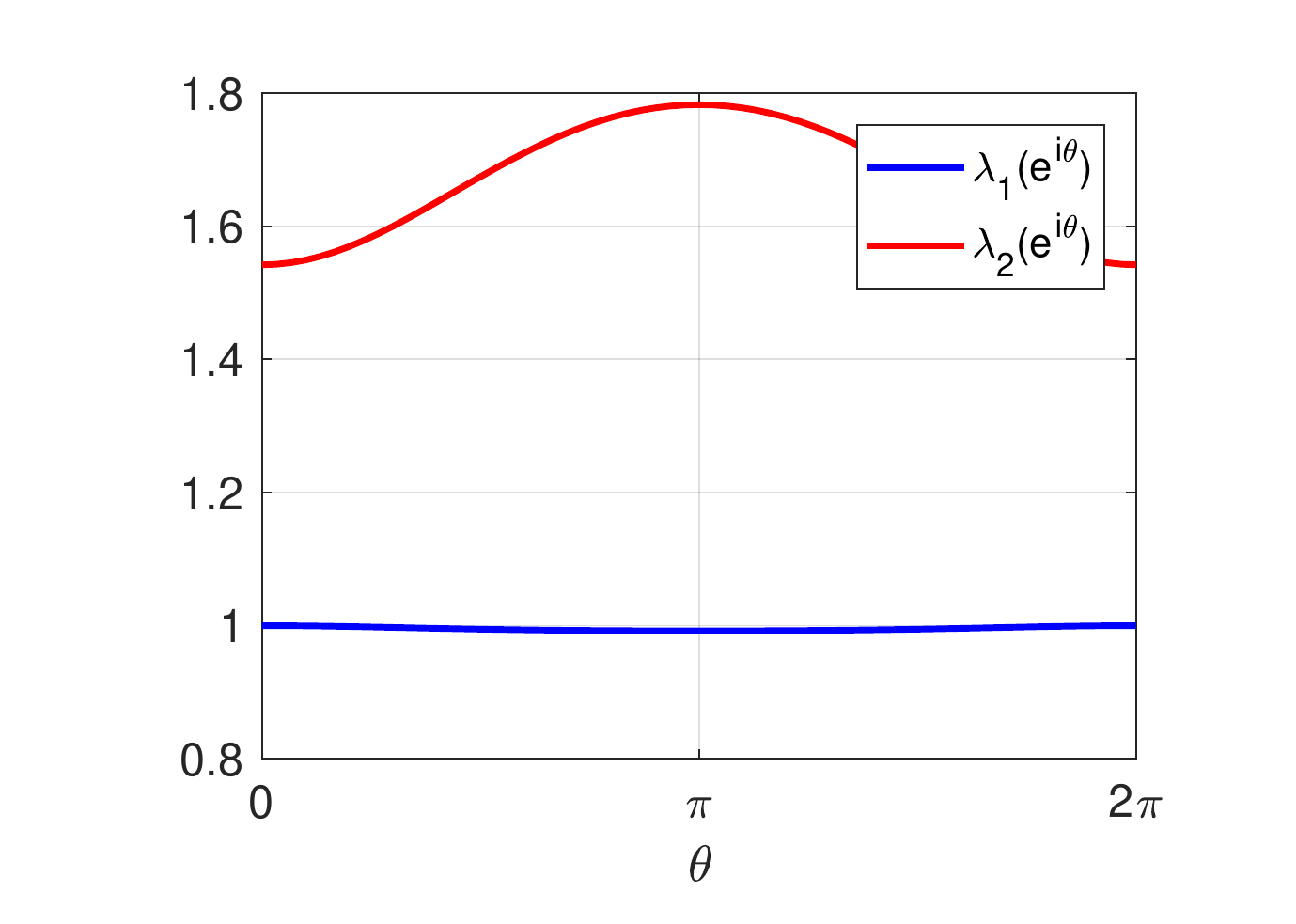}}
  \caption{The eigenvalues $\lambda_{1,2}(e^{\ii\theta})$ for $u_p=u_{p,1}$ in (a), $u_p=u_{p,2}$ in (b) and 
  $u_p=u_{p,3}$ in (c).  Here $\omega$ is given as in \eqref{eq:exp-exp} with $S_1=4, s_1=2$, $S_2=1.5$, $s_2=1,$ $h=0.4,$ and $T=3.2.$  The corresponding spectra are $\sigma(L)=[0.9969, 1]\cup[3.1147, 3.4945]$ , $\sigma(L)=[0.8020, 0.9692]\cup[0.9978, 1.0022]$  and $\sigma(L)=[0.9921, 1]\cup[1.5419, 1.7825]$, respectively. (All the approximated values are rounded up to 4 decimals).}\label{spec2b}
\end{figure}

\begin{figure}[H]
  \centering
  \includegraphics[width=4.5cm]{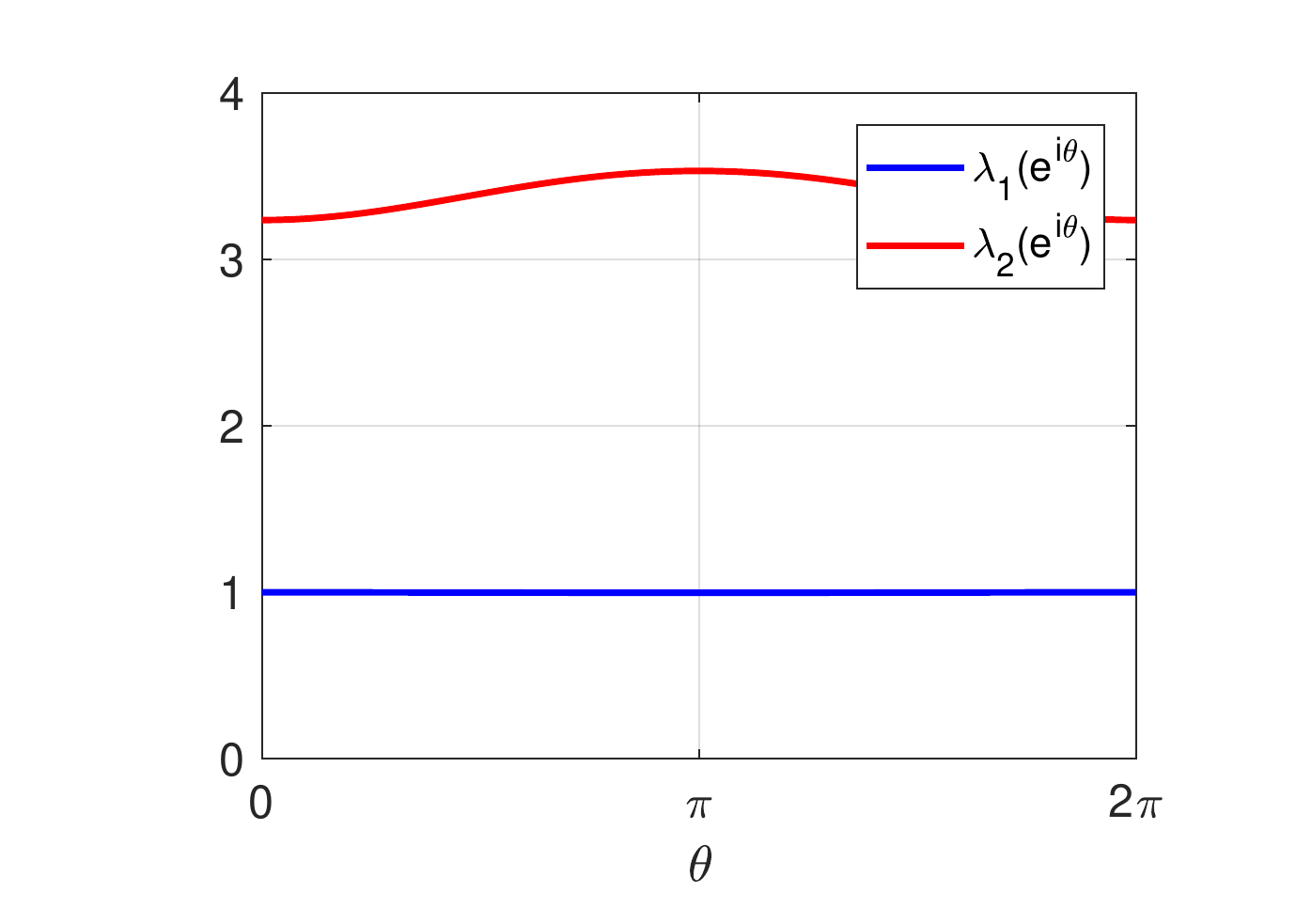}
  \includegraphics[width=4.5cm]{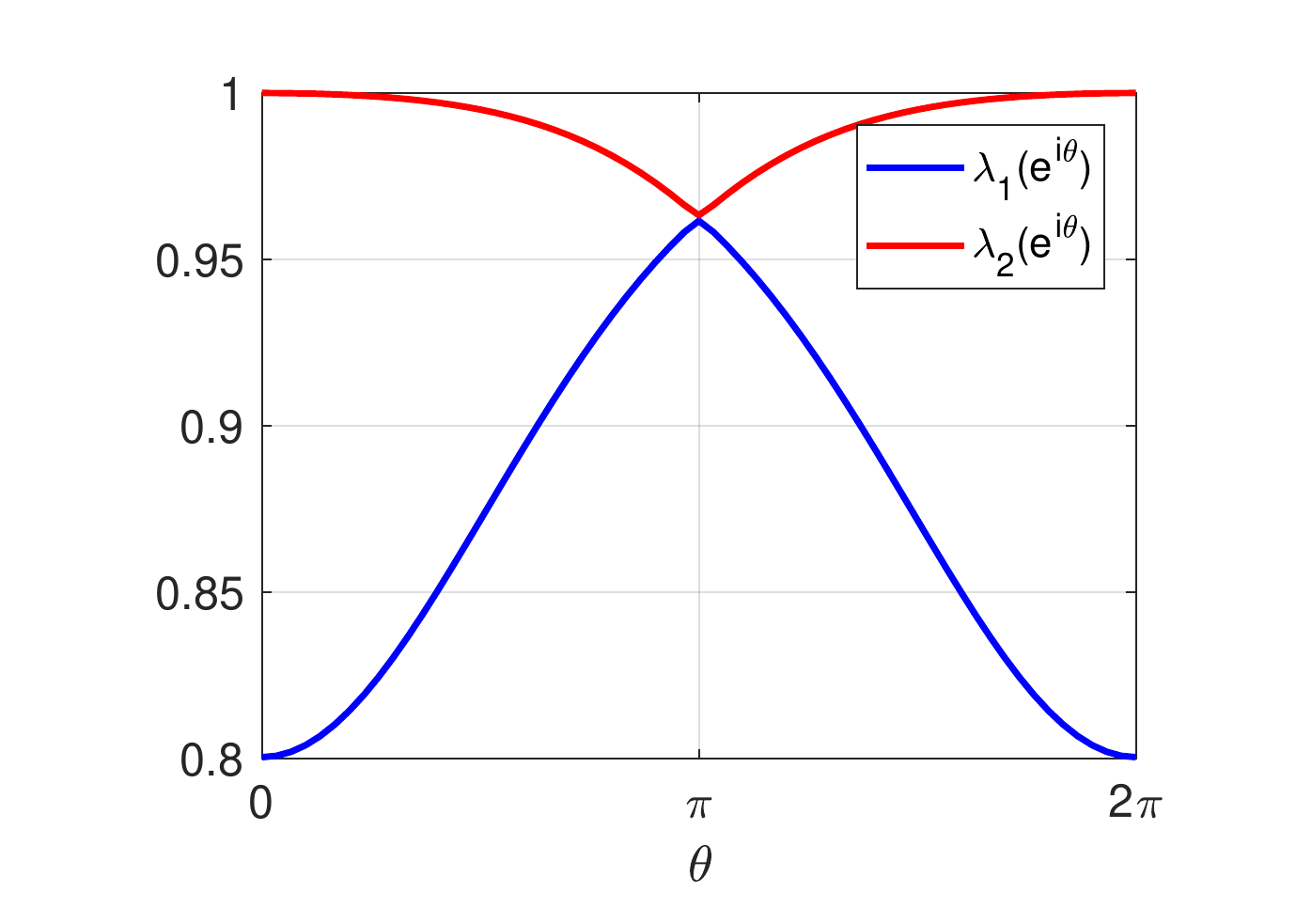}
  \includegraphics[width=4.5cm]{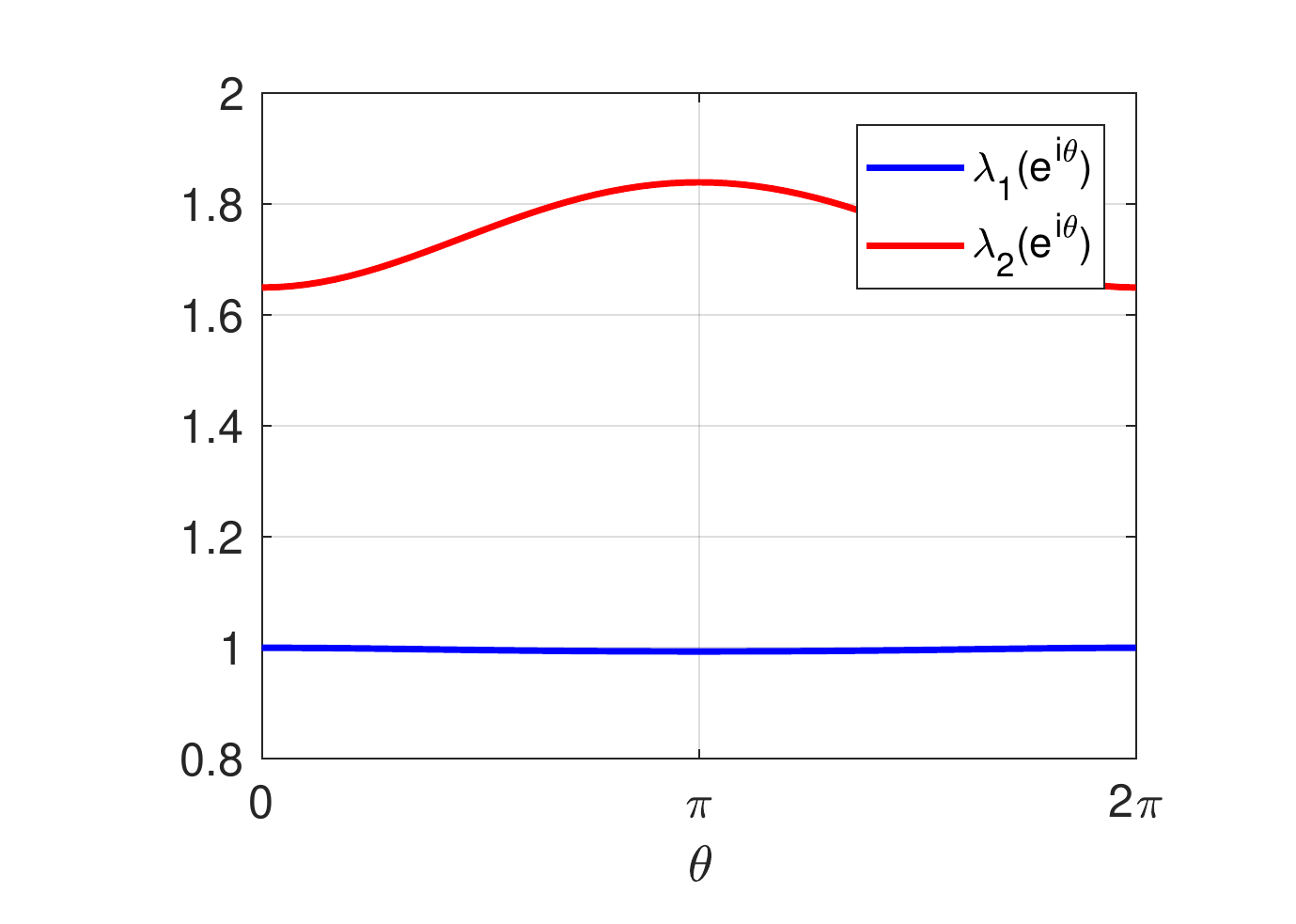}
    \caption{
    The eigenvalues $\lambda_{1,2}(e^{\ii\theta})$ for $u_p=u_{p,1}$ in (a), $u_p=u_{p,2}$ in (b) and
  $u_p=u_{p,3}$ in (c).  Here $\omega$ is given as in \eqref{eq:exp-exp} with $S_1=4, s_1=2$, $S_2=1.5$, $s_2=1,$ $h=0.4,$ and $T=3.5.$  The corresponding spectra are $\sigma(L)=[0.9973, 1]\cup[3.2365, 3.5318]$, $\sigma(L)=[0.8005, 0.9616]\cup[0.9633, 1]$  and $\sigma(L)=[0.9934, 1]\cup[1.6494, 1.8390]$, respectively. (All the approximated values are rounded up to 4 decimals).}\label{spec2c}
\end{figure}

\section{Conclusions and outlook}\label{Sec:Conclusions}

In most cases  $\omega_p$ has no analytic expression and has to be approximated.
This may lead to some difficulties in analysing the behaviour of $W_p$ on $[0, T/2]$ needed for the existence analysis and calculating the symbol $\Phi$. However, the considered approach of constructing periodic solutions is still simpler than the ODE method proposed in \cite{Krisner2007PeriodicSolutions} and, in addition, allows to address the uniqueness of solutions. Moreover, it is not restricted to a particular type of $\omega$ as in \cite{Krisner2007PeriodicSolutions}. The downside of the approach is that, at the moment, we have restricted the choice  of $f$ to the Heaviside function. This choice however allowed us to analyse linear stability of the solutions which to the best of our knowledge has not been addressed up till now.
We have shown that \eqref{Amari_model} can posses both linearly stable and unstable periodic solutions.  We conjecture that the existence and stability results hold for steep sigmoid like functions $f$, see \eqref{Eq:FiringRateFun}. To prove this conjecture we plan to proceed in the way similar to \cite{oleynik2016spatially} and \cite{OKS(Unpublished)}. It is not possible to apply the results from  the mentioned papers directly here since the eigenvalue $\lambda=1$ of $\cH'(u_p)$ is not isolated. However, the stability analysis in Section \ref{Sec:Stability} shows that the spectrum is pointwise and the eigenfunctions could be calculated, which gives a possibility of studying the dynamics of solutions on a central manifold. We plan to address this problem in our future work.

Another topic, that we have not properly addressed in this paper, is the coexistence of the localized and periodic solutions with different stability properties. The combination of the ODE methods \cite{ELVIN2010537,krisner2004ODE,DEVANEY1976431} with the results obtained here could be used to address this interesting problem.

Finally, we would like to mention that the analysis presented here could be generalized to the case of $N$-bump periodic solutions and several dimensions.

\section{Acknowledgements}

The authors are grateful to Professor Arcady Ponosov and John Wyller (Norwegian University of Life Sciences) for fruitful discussions during the preparation phase of this paper.
 We thank Professor M{\aa}rten Gulliksson (\"{O}rebro University) for constructive criticism of the manuscript.
This research work was supported by the Norwegian University of Life
Sciences and The Research Council of Norway, project number 239070.

\bibliographystyle{unsrt}

\end{document}